\documentclass[12pt]{amsart}

\usepackage{etex}
\usepackage{amsmath, amssymb}
\usepackage{array}
\usepackage[frame,cmtip,arrow,matrix,line,graph,curve]{xy}
\usepackage{graphpap, color, paralist, pstricks}
\usepackage[mathscr]{eucal}
\usepackage[pdftex]{graphicx}
\usepackage[pdftex,colorlinks,backref=page,citecolor=blue]{hyperref}
\usepackage{tikz}
\usepackage{subcaption}
\usepackage{rotating}
\usepackage{comment}

\numberwithin{equation}{section}
\numberwithin{figure}{section}

\newtheorem{theorem}{Theorem}[section]
\newtheorem{theoremletter}{Theorem}

\newtheorem{proposition}[theorem]{Proposition}

\newtheorem{lemma}[theorem]{Lemma}
\newtheorem*{lemmaletter}{Compatibility Lemma}

\newtheorem{conjecture}[theorem]{Conjecture}

\theoremstyle{definition}
\newtheorem{definition}[theorem]{Definition}
\newtheorem{example}[theorem]{Example}

\newtheorem{question}[theorem]{Question}
\newtheorem{remark}[theorem]{Remark}

\newtheorem*{acknowledgement}{Acknowledgement}

\newcommand{\CC}{\mathbb{C} }

\newcommand{\ZZ}{\mathbb{Z} }
\newcommand{\QQ}{\mathbb{Q} }
\newcommand{\RR}{\mathbb{R} }

\newcommand{\cA}{\mathcal{A} }
\newcommand{\cC}{\mathcal{C} }

\newcommand{\cF}{\mathcal{F} }

\newcommand{\cS}{\mathcal{S} }
\newcommand{\cT}{\mathcal{T} }
\newcommand{\cU}{\mathcal{U} }

\newcommand{\ba}{\mathbf{a} }
\newcommand{\bfb}{\mathbf{b} }

\newcommand{\be}{\mathbf{e} }

\newcommand{\bx}{\mathbf{x} }
\newcommand{\by}{\mathbf{y} }
\newcommand{\bA}{\mathbf{A} }
\newcommand{\bE}{\mathbf{E} }

\def\SL{\mathrm{SL}}

\def\spec{\mathrm{Spec}\, }

\def\im{\mathrm{im}}

\begin{document}

\title[Consequences of the compatibility of skein algebra and cluster algebra]{Consequences of the compatibility of skein algebra and cluster algebra on surfaces}
\date{\today}

\author{Han-Bom Moon}
\address{Department of Mathematics, Fordham University, New York, NY 10023}
\email{hmoon8@fordham.edu}

\author{Helen Wong}
\address{Department of Mathematical Sciences, Claremont McKenna College, Claremont, CA 91711}
\email{hwong@cmc.edu}

\begin{abstract}
We investigate two algebra of curves on a topological surface with interior punctures -- the cluster algebra of surfaces defined by Fomin, Shapiro, and Thurston, and the generalized skein algebra constructed by Roger and Yang. By establishing their compatibility, we resolve Roger-Yang's conjecture on the deformation quantization of the decorated Teichm\"uller space. We also obtain several structural results on the cluster algebra of surfaces. The cluster algebra of a positive genus surface is not finitely generated, and it differs from its upper cluster algebra. 
\end{abstract}

\maketitle

\section{Introduction}\label{sec:introduction}

By a surface $\Sigma_{g, n}$, we denote a compact Riemann surface of genus $g$, without boundaries, minus $n$ punctures. We may associate two `algebras of curves' on $\Sigma_{g, n}$, but coming from entirely different motivations -- one from geometric topology,  and the other from combinatorial algebra. In this paper, we establish compatibility between the two algebras. By employing it, we prove several structural results about each of the algebras that might not be readily apparent if considering each algebra separately.  

The first algebra is  the \emph{curve algebra} $\cC(\Sigma_{g, n})$,  which belongs to a family of invariants of surfaces that are related to the Jones polynomial for knots \cite{Jones85} and to the Witten-Reshetikhin-Turaev topological quantum field theory \cite{Wit89, RT91, BHMV95}. The most well-studied of this family is the Kauffman bracket skein algebra of an unpunctured surface  \cite{Przytycki91, Turaev91}.  It is known to be related to hyperbolic geometry--- the skein algebra is the deformation quantization of the  $\SL_{2}$-character variety, which contains the Teichm\"uller space of the surface \cite{Turaev91, Bullock97, BullockFrohmanJKB99, PrzytyckiSikora00}. 
In \cite{RogerYang14}, Roger and Yang sought to generalize this relationship between the skein algebra and the Teichm\"uller space to the case of a punctured surface.  For a punctured surface $\Sigma_{g, n}$,  they defined a generalized skein algebra $\cS^{q}(\Sigma_{g, n})$ 
spanned by disjoint unions of framed knots, arcs, and vertex classes. They proposed that it should be the deformation quantization for the decorated Teichm\"uller space $\cT^{d}(\Sigma_{g, n})$ constructed by Penner in \cite{Penner87, Penner92}.  The curve algebra $\cC(\Sigma_{g, n})$ we study in this paper is the classical limit of Roger-Yang's generalized skein algebra $\cS^{q}(\Sigma_{g, n})$ obtained by setting $q = 1$.

The second algebra studied in this paper is the \emph{cluster algebra} $\cA(\Sigma_{g, n})$ of a surface. Such cluster algebras were observed by \cite{GekhtmanShapiroVainshtein05, FominShapiroThurston08} to be interesting examples of the cluster algebras originally introduced by Fomin and Zelevinsky in \cite{FominZelevinsky02} for studying the total positivity and dual canonical bases in Lie theory. Defined for any punctured surface $\Sigma_{g, n}$ admitting an ideal triangulation, the cluster algebra $\cA(\Sigma_{g, n})$ is generated by arcs with \emph{tagging} (of plain or notched) at its endpoints. 
From the combinatorial perspective, the tagging is needed for $\cA(\Sigma_{g, n})$ to have the structure of a cluster algebra,
and various geometric interpretations can be seen from \cite{FockGoncharov06, MusikerSchifflerWilliams11, FominThurston18, AllegrettiBridgeland20}. 
This paper was started in part from the authors' attempt to better understand the relationship between the two algebras.  

Each of the two algebra has its own distinct features, and the main results of this paper follow from transferring advantageous properties from one algebra to another.  Our primary tool is an injective homomorphism from the curve algebra $\cC(\Sigma_{g, n})$ to the cluster algebra $\cA(\Sigma_{g, n})$, which manifests the `compatibility' of the two commutative algebras.   By leveraging the integrality of $\cA(\Sigma_{g, n})$, we prove that the generalized skein algebra $\cS^{q}(\Sigma_{g, n})$ is a deformation quantization of $\cT^{d}(\Sigma_{g, n})$ and resolve Roger-Yang's conjecture (Theorem \ref{thm:defquantTeich}).  Compatibility also enables us to define a nontrivial `reduction' map and prove the non-finite generation of $\cA(\Sigma_{g,n})$ (Theorem \ref{thm:infinitegenerationA}) for $g \ge 1$. The unifying theme of this paper is the interplay between the two algebras afforded by compatibility. 

\subsection{Compatibility of curve algebra and cluster algebra}\label{ssec:compatibility}

Let $\Sigma_{g, n}$ be a Riemann surface of genus $g$ with $n>0$ punctures. We assume that $\chi(\Sigma_{g, n}) < 0$, so that  $n$-punctured spheres with $n = 1, 2$ are excluded. When we define the cluster algebra, we also exclude the three-punctured sphere.

We have an explicit comparison of $\cA(\Sigma_{g, n})$ and $\cC(\Sigma_{g, n})$, which will be the key step to the main results of this paper. 

\begin{lemmaletter}\label{lem:compatibility}
Let $\cA(\Sigma_{g, n})$ be the cluster algebra and $\cC(\Sigma_{g, n})$ be the curve algebra associated to $\Sigma_{g, n}$. Then there is a monomorphism 
\[
	\rho : \cA(\Sigma_{g, n}) \to \cC(\Sigma_{g, n}). 
\]
\end{lemmaletter}

This is not merely an existence statement. As discussed in Section \ref{sec:rho}, the construction of $\rho$ gives a simple geometric interpretation of the tagging, which can be plain or notched (Definition \ref{def:rholocal}). The upshot is that ``a notch is a vertex class,'' where a vertex class is a formal variable in $\cC(\Sigma_{g, n})$ assigned to each puncture (Definition \ref{def:curvealgebra}).  

Both algebras have their geometric origin from the same decorated Teichm\"uller space due to Penner, and the proof of the Compatibility Lemma  is relatively straightforward (see Section \ref{sec:rho}). Indeed, a similar compatibility result for surfaces with boundaries but without interior punctures was proven by Muller in \cite{Muller16}. We note that in  because there are no interior punctures, the vertex classes in the skein algebra and the tagged arcs in the cluster algebra do not exist, and thus the interpretation of the tagging as a vertex is new in the punctured surface case. 

Like other recent developments \cite{BazierMatteSchiffler22, HikamiInoue15, LeeSchiffler19, NagaiTerashima20, Yacavone19}, one could think of these compatibility results as another indication of deep connection between knot theory and cluster algebra.   Instead, what we would rather emphasize here is the key role the Compatibility Lemma plays in the proofs of the main results of this paper, as we will see throughout the paper.  Let us now describe the main results in this paper, beginning first with the quantum theory of skein algebras and then turning to the structural theory of cluster algebras.

\subsection{Skein algebra and deformation quantization}

In \cite{RogerYang14}, Roger and Yang introduced a generalized skein algebra $\cS^{q}(\Sigma_{g, n})$ as a candidate of the deformation quantization of $\cT^{d}(\Sigma_{g, n})$. Their program consists of two steps. First, they showed that $\cS^{q}(\Sigma_{g, n})$ is a deformation of quantization of its classical limit $\cC(\Sigma_{g, n})$. They then proved that there is a Poisson algebra homomorphism $\Phi : \cC(\Sigma_{g, n}) \to C^{\infty}(\cT^{d}(\Sigma_{g, n}))$ whose Poisson structures are given by the generalized Goldman bracket and the Weil-Peterssen form, respectively \cite[Theorem 1.2]{RogerYang14}. Thus, if the Poisson algebra representation is faithful (meaning $\Phi$ is injective), then $\cS^{q}(\Sigma_{g, n})$ can be understood as the quantization of $\cT^{d}(\Sigma_{g, n})$.   However, they left the faithfulness as a conjecture \cite[Conjecture 3.4]{RogerYang14}. In Section \ref{sec:impcurve}, we prove it  by employing Compatibility Lemma and finish Roger and Yang's program.

\begin{theoremletter}\label{thm:defquantTeich}
The Roger-Yang generalized skein algebra $\cS^{q}(\Sigma_{g, n})$ is a deformation quantization of the decorated Teichm\"uller space $\cT^{d}(\Sigma_{g, n})$. 
\end{theoremletter}

Moreover, a consequence of our proof of Theorem \ref{thm:defquantTeich} is that the fractional algebras of both the two algebras $\cA(\Sigma_{g,n})$ and $\cC(\Sigma_{g, n})$ are identical. It thus follows that

\begin{theoremletter}\label{thm:defquantcluster}
The Roger-Yang generalized skein algebra $\cS^{q}(\Sigma_{g, n})$ is a deformation quantization of $\cA(\Sigma_{g, n})$. 
\end{theoremletter}

Note that in our earlier paper \cite{MoonWong21}, we already showed that Theorem \ref{thm:defquantTeich} hold when $n$ is relatively large compared to $g$ \cite[Theorem B]{MoonWong21}. The proof was based on a long diagramatical computation with little theoretical support nor intuition. We find that our proofs here, based on the relationship with cluster algebras established by Compatibility Lemma, provides a more satisfactory theoretical reasoning. 

\begin{remark}
In \cite{Muller16}, Muller has a similar result as Theorem \ref{thm:defquantcluster} in the case of surfaces with boundaries but without punctures, resulting in a quantization where arcs $q$-commute.  However, Muller's method cannot apply in the case considered in this paper, because $\cA(\Sigma_{g, n})$ does not extend to a quantum cluster algebra of Berenstein-Zelevinski \cite{BerensteinZelevinsky05} if there is an interior puncture (see Remark \ref{rmk:quantumclusteralgebra} for a more in depth discussion). Instead, the quantization of Theorems \ref{thm:defquantTeich} and \ref{thm:defquantcluster} use the Poisson structure for the Roger-Yang skein algebra based on Mondello's computation for $\lambda$-length of arcs \cite{Mondello09}.  In particular, these $\lambda$-lengths do not form log canonical coordinates in the sense of \cite[Section 2.2]{GekhtmanShapiroVainshtein05}.   Hence arcs are not $q$-commutative in the quantization, but satisfy a two-term skein relation that generalizes the Ptolemy exchange relations for cluster variables.
\end{remark}

\subsection{Comparison of cluster algebras with their upper cluster algebra}
 
The \emph{upper cluster algebra} $\cU$ (Definition \ref{def:upperclusteralgebra}) contains the ordinary cluster algebra $\cA$ and is constructed from the same combinatorial data of seed. In many ways, $\cU$  behaves better than $\cA$, and thus the question of whether $\cA = \cU$ or not has attracted many researchers in the cluster algebra community. For the summary of some known results, see \cite[Section 1.2]{CanakciLeeSchiffler15} and a very recent result \cite{IshibashiOyaShen23}. For $\cA(\Sigma_{g, n})$, when $n = 1$, it was shown that $\cA(\Sigma_{g, 1}) \ne \cU(\Sigma_{g, 1})$ by Ladkani \cite{Ladkani13}. 

Here, we use the curve algebra $\cC(\Sigma_{g, n})$ and a variation $\cC(\Sigma_{g, n})'$ (Definition \ref{def:Cprime}), and obtain an inclusion 
\begin{equation}\label{eqn:comparison}
	\cA(\Sigma_{g, n}) \subset \cC(\Sigma_{g, n})' \subset \cU(\Sigma_{g, n}).
\end{equation}
The algebra $\cC(\Sigma_{g, n})'$ is a subalgebra of $\cC(\Sigma_{g, n})$ generated by the image of $\cA(\Sigma_{g, n})$ and the Kauffman bracket skein algebra of $\Sigma_{g, n}$ generated by isotopy classes of loops. We conjecture that $\cC(\Sigma_{g, n})' = \cU(\Sigma_{g, n})$ (Conjecture \ref{conj:CandU}). However, to the authors' knowledge, it is still unknown if the `geometric' subalgebra $\cC(\Sigma_{g, n})'$ generated by tagged arcs and loops coincide with $\cU(\Sigma_{g, n})$. See Remark \ref{rem:midalgebra}. For the comparison of $\cC(\Sigma_{g, n})$ and $\cU(\Sigma_{g, n})$, see Remark \ref{rem:CandU}.

\subsection{Determining whether cluster algebras are finitely generated}

It is known that the Roger-Yang skein algebra is finitely generated \cite{BKPW16JKTR}, and our method is to use the compatibility map $\rho$ to deduce results about the cluster algebra. 

By \cite{Ladkani13}, it is known that $\cA(\Sigma_{g, 1})$ is not finitely generated for all $g \ge 1$. We prove the following:

 \begin{theoremletter}\label{thm:infinitegenerationA}
The cluster algebra  of a sphere $\cA(\Sigma_{0, n})$ is finitely generated. On the other hand, for $g \ge 1$, $\cA(\Sigma_{g, n})$ is not finitely generated. 
\end{theoremletter}

Note that the cluster algebra is defined only when a surface has punctures, so $n\geq 1$ for all cases.  And in the case of a sphere, we additionally require $n \geq 3$ punctures. From Theorem \ref{thm:infinitegenerationA} and \eqref{eqn:comparison}, we have the immediate corollary:

\begin{theoremletter}\label{thm:torusupper}
For every $g \ge 1$, $\cA(\Sigma_{g, n}) \ne \cU(\Sigma_{g, n})$. 
\end{theoremletter} 

We would also like to note that the case of $g= 0$ is exceptional.  Indeed, we expect that $\cA(\Sigma_{0, n}) = \cU(\Sigma_{0, n})$ (Conjecture \ref{conj:CandU}). For instance, when $g = 0$, one can show that loops are also in $\cA(\Sigma_{0, n})$, by adapting the computation in \cite{BKPW16Involve} and \cite{ACDHM21} (Remark \ref{rem:g=0AC}).

\subsection{Structure of the paper}

Sections \ref{sec:arcalgebra} and \ref{sec:clusteralgebrasurface} are review of the definition and basic properties of $\cC(\Sigma_{g, n})$ and $\cA(\Sigma_{g,n})$, respectively, and related constructions. In Section \ref{sec:rho}, we start with the compatibility map $\rho$, and show that is well-defined and injective.  The next two sections detail our main results---Section \ref{sec:impcurve} establishes the curve algebra as a quantization of decorated Teichmuller space, and Section \ref{sec:clusteralgebraimplication} discusses algebraic properties of the cluster algebra and upper cluster algebra. 

\acknowledgement

The authors would like to thank to Wade Bloomquist, Hyunkyu Kim, Thang Le, Kyungyong Lee, Gregg Musiker, Fan Qin, and Dylan Thurston for valuable conversations. This work was completed while the first author was visiting Stanford University. He gratefully appreciates the hospitality during his visit.  The second author is partially supported by grant DMS-1906323 from the US National Science Foundation and a Birman Fellowship from the American Mathematical Society. 


\section{The curve algebra $\cC(\Sigma_{g,n})$}\label{sec:arcalgebra}

In this section, we give a formal definition and basic properties of the curve algebra $\cC(\Sigma_{g, n})$. For details, see \cite[Section 2.2]{RogerYang14} and \cite[Section 2.4]{MoonWong21}. 

In this paper, a \emph{surface} is  $\Sigma_{g,n} := \overline{\Sigma}_{g}\setminus V$, where $\overline{\Sigma}_{g}$ is a Riemann surface of genus $g$ without boundary, and $V = \{v_{1}, \ldots, v_{n}\}$ is a finite set of points in $\overline{\Sigma}_{g}$. We call $V$ the set of \emph{punctures} or \emph{vertices}.   

 A \emph{loop} $\underline{\alpha}$ on $\Sigma_{g, n}$ is an immersion of a circle into $\Sigma_{g, n}$. An \emph{arc} $\underline{\beta}$ in $\Sigma$ is an immersion of $[0, 1]$ into $\overline{\Sigma}_{g}$ such that the image of $(0, 1)$ is in $\Sigma_{g, n}$ and the image of two endpoints are (not necessarily distinct) points in $V$. The seemingly unnecessary underbar notation will be justified in Section \ref{sec:clusteralgebrasurface}.

\begin{definition}\label{def:curvealgebra}
Let $R$ be a commutative ring. The \emph{curve algebra $\cC(\Sigma_{g, n})_{R}$} is the $R$-algebra generated by isotopy classes of loops,  arcs, $V = \{v_{i}\}$, and their formal inverses $\{v_{i}^{-1}\}$, modded out by the following relations:
\begin{center}
\begin{tabular}{lll}
(1)& (Skein relation) & 
$\begin{minipage}{.4in}\includegraphics[width=\textwidth]{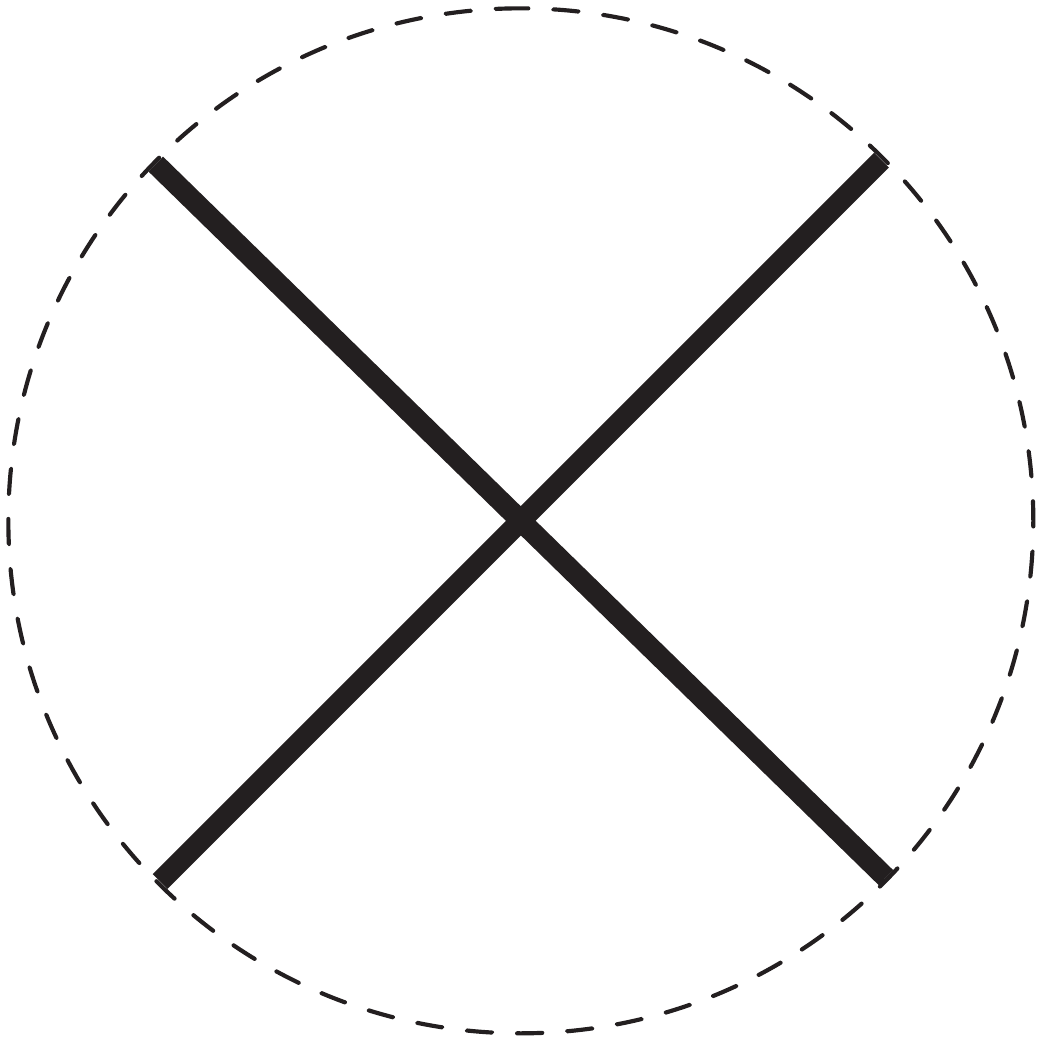}\end{minipage} 
-  \left( \begin{minipage}{.4in}\includegraphics[width=\textwidth]{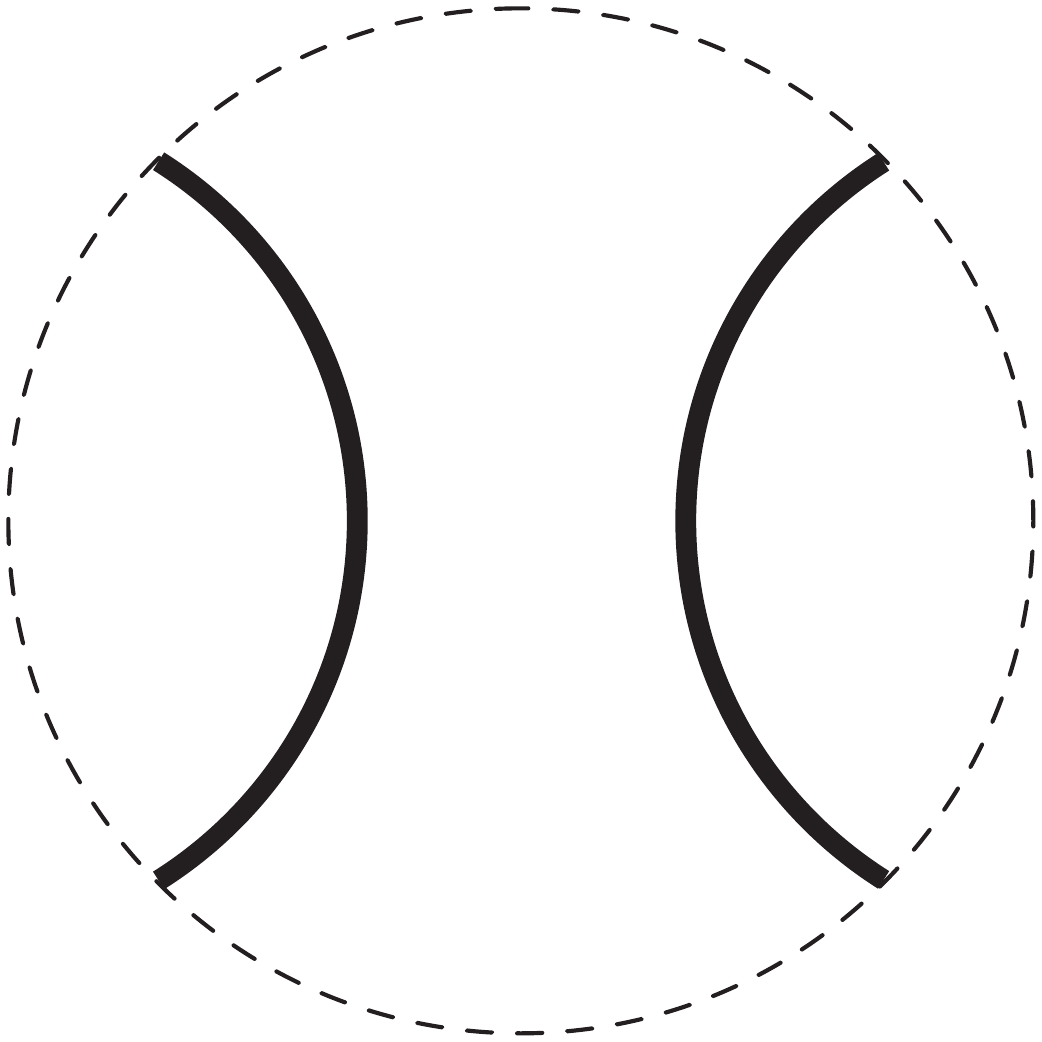}\end{minipage} 
+\begin{minipage}{.4in}\includegraphics[width=\textwidth]{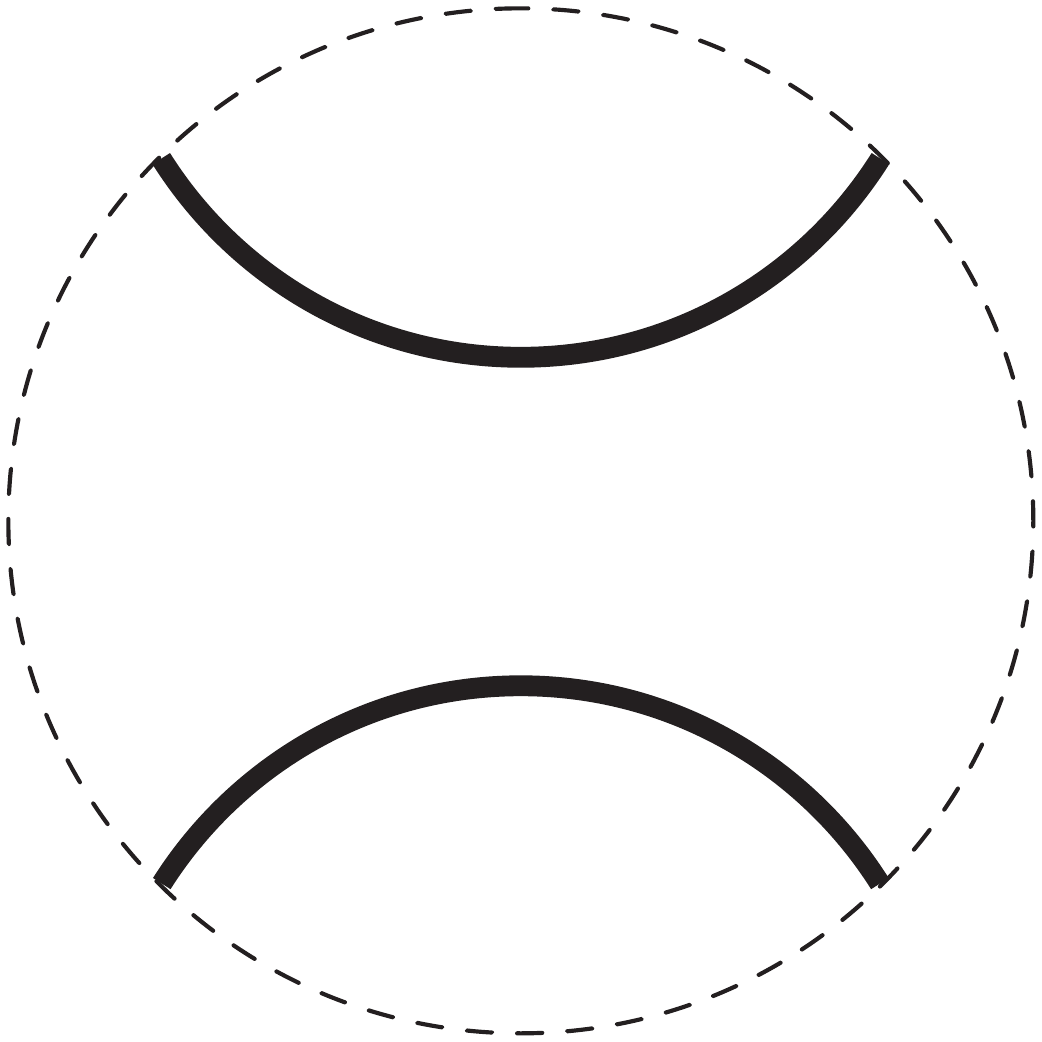}\end{minipage}  \right)$\\
(2) & (Puncture-skein relation) &
$v_i \begin{minipage}{.4in}\includegraphics[width=\textwidth]{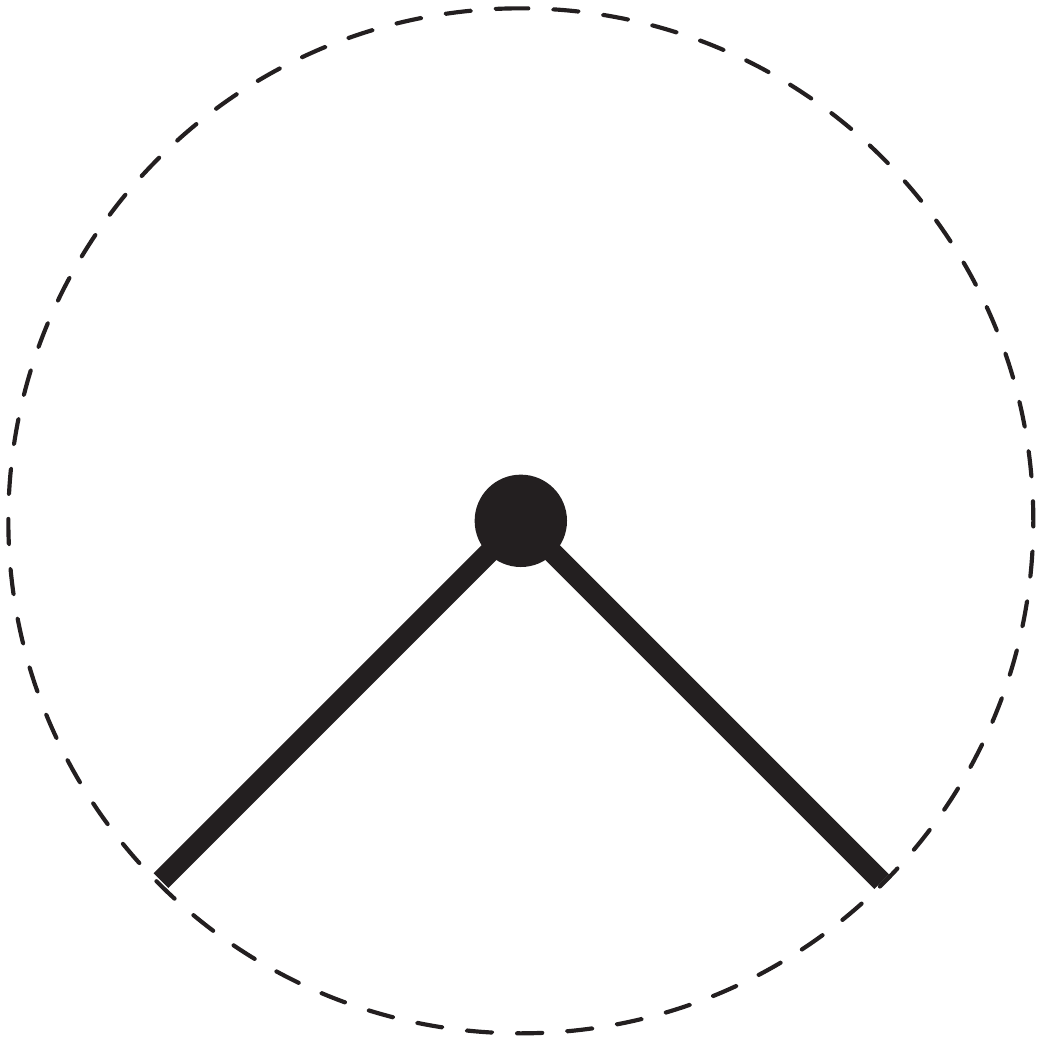}\end{minipage} 
-  \left( \begin{minipage}{.4in}\includegraphics[width=\textwidth]{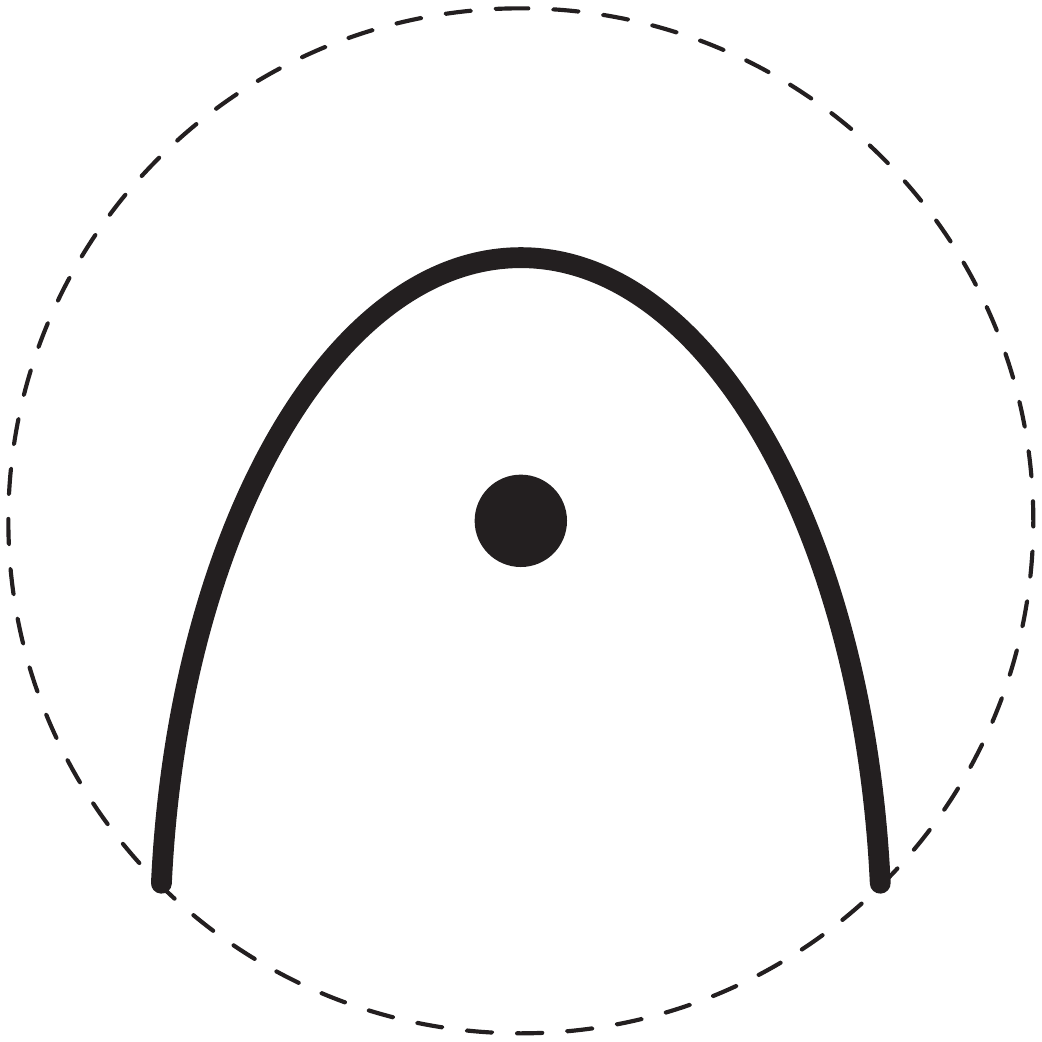}\end{minipage} 
+\begin{minipage}{.4in}\includegraphics[width=\textwidth]{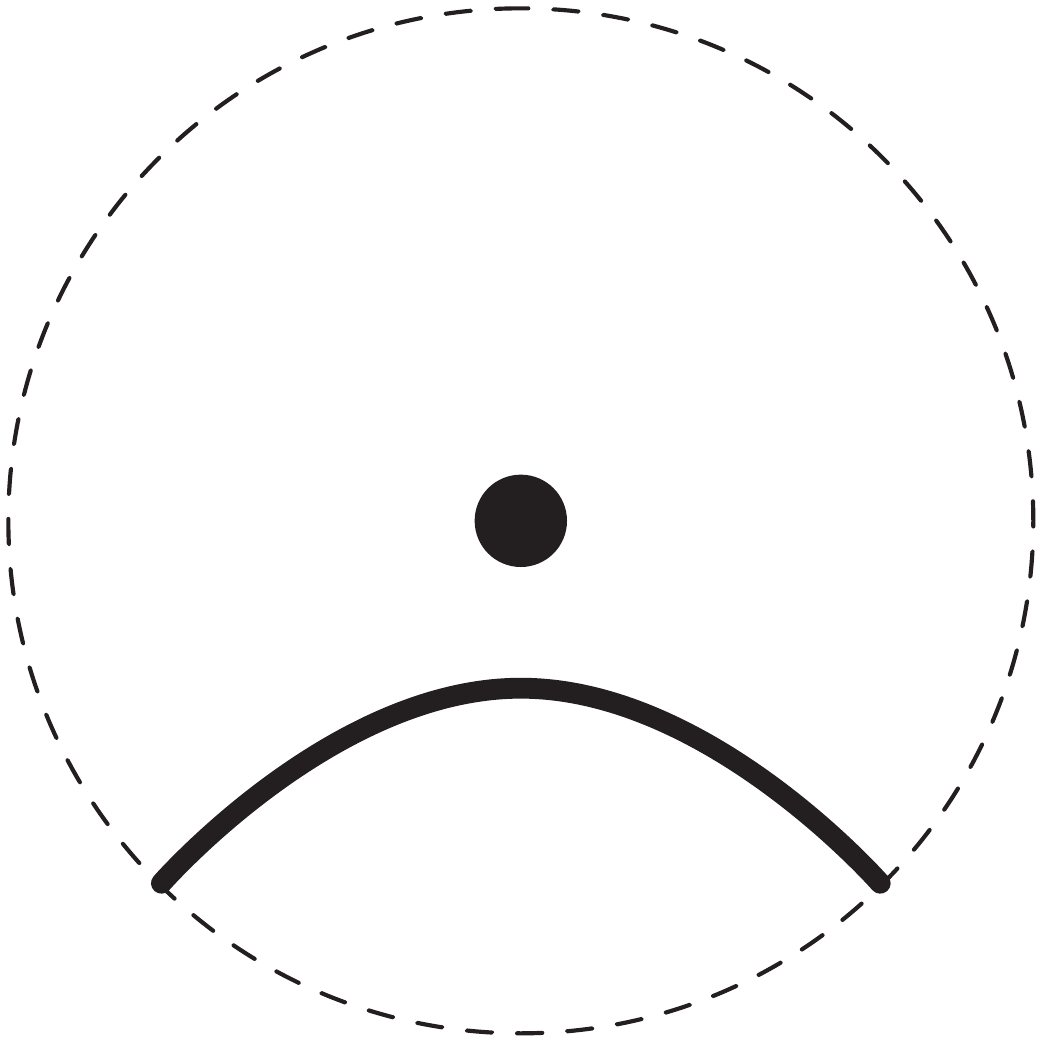}\end{minipage}  \right)$\\
(3)& (Framing relation) & 
$\begin{minipage}{.4in}\includegraphics[width=\textwidth]{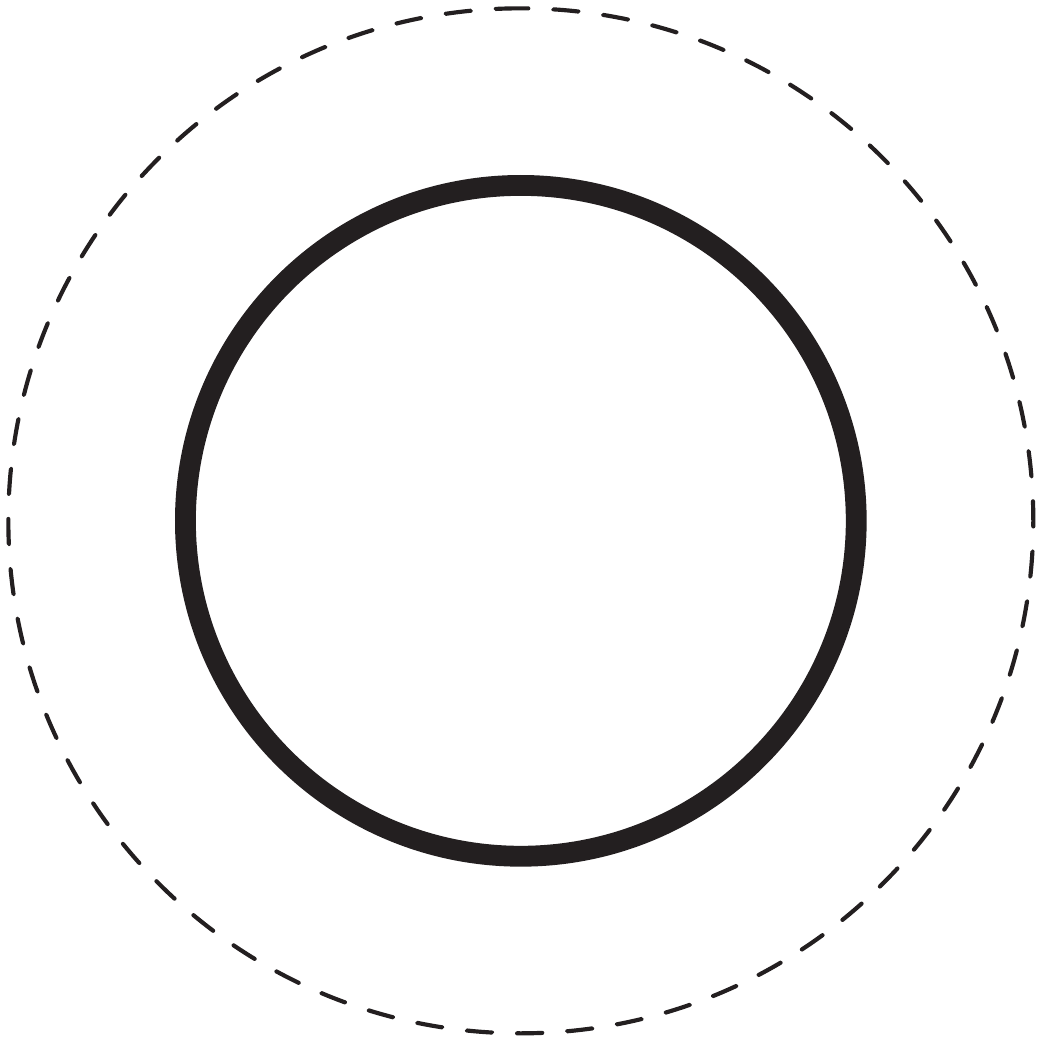} \end{minipage} 
+ 2$\\
(4) & (Puncture-framing relation) & 
$\begin{minipage}{.4in}\includegraphics[width=\textwidth]{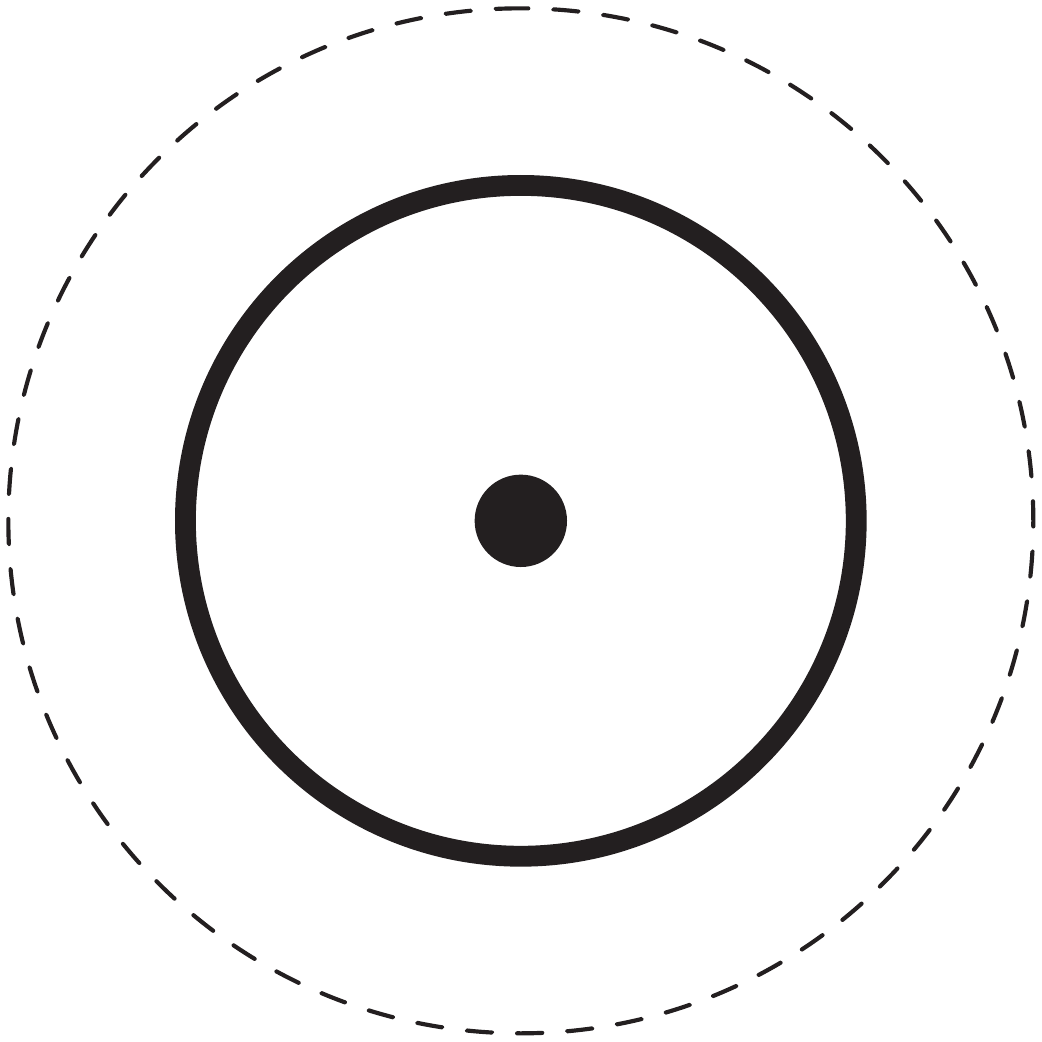} \end{minipage} 
-2$.
\end{tabular}
\end{center}

The multiplication of elements in $\cC(\Sigma_{g, n})_R$ are represented by taking the union of generators (and counted with multiplicity). We allow the empty curve $\emptyset$ and it is the multiplicative identity. In the relations, the curves are assumed to be identical outside of the small balls depicted, and the $i$-th puncture $v_i$ is depicted in the second relation. 

We set $\cC(\Sigma_{g, n}) := \cC(\Sigma_{g, n})_{\ZZ}$, so that we mean $R = \ZZ$ by default. Then $\cC(\Sigma_{g, n})_{R} = \cC(\Sigma_{g, n}) \otimes_{\ZZ}R$. 
\end{definition}

\begin{remark}
Note that in the curve algebra originally discussed by Roger-Yang \cite{RogerYang14},  they used  $R = \CC$, and  the vertices $\{v_{i}\}$ were treated as coefficients. But for our purpose, it is more natural to think of the vertices as  generators of the algebra. 
\end{remark}

 \begin{remark}\label{rem:torsionfree}
One might wonder about our choice of coefficient ring $\ZZ$, as compared to Roger and Yang's choice of  $\CC$. Clearly, there is a morphism $\cC(\Sigma_{g, n}) \to \cC(\Sigma_{g, n})_{\CC}$.  In addition, one can adapt the proof of \cite[Theorem 2.4]{RogerYang14} by replacing the $\CC$-coefficient by the $\ZZ$-coefficient to show that $\cC(\Sigma_{g, n})$ (with $\ZZ$-coefficients) has no torsion. Thus, we have an inclusion $\cC(\Sigma_{g, n}) \subset \cC(\Sigma_{g, n})_{\CC}$. It follows, for example, that if $\cC(\Sigma_{g, n})_{\CC}$ is an integral domain, then $\cC(\Sigma_{g, n})$ is also an integral domain. 
\end{remark}

\begin{example}\label{ex:unpuncturedmonogon}
Let $\alpha$ be an arc bounding an unpunctured monogon with the vertex $v$.  Then we can compute $v \alpha$ as follows: 
\[ v \left( \begin{minipage}{.5in}\includegraphics[width=\textwidth]{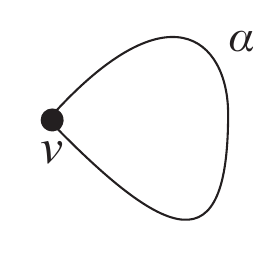}\end{minipage} \right) 
=
\begin{minipage}{.5in}\includegraphics[width=\textwidth]{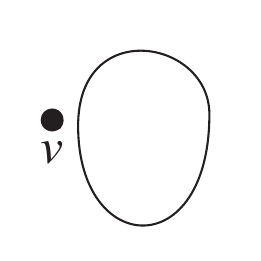}\end{minipage}
+
\begin{minipage}{.5in}\includegraphics[width=\textwidth]{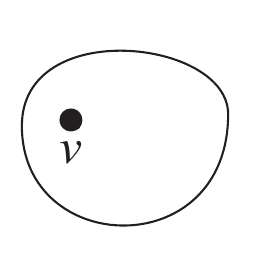}\end{minipage} 
= -2 + 2 = 0
\]
Since $v$ is invertible in $\cC(\Sigma_{g, n})$, this shows that any arc bounding an unpunctured monogon is zero in $\cC(\Sigma_{g, n})$.
\end{example}

\begin{lemma}\label{lem:monogon}
Let $v$ and $w$ be two distinct punctures, and $e$ be an arc connecting $v$ and $w$. In addition, let $\gamma$ be an arc with both ends at $v$ that bounds a one-punctured monogon containing $w$.  Then $\gamma = we^{2}$. 
\end{lemma}

\begin{proof}
\[
	w e^2 = w \left( \begin{minipage}{1in}\includegraphics[width=\textwidth]{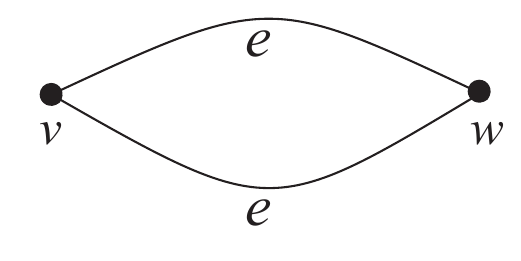}\end{minipage} \right) = 
\left( \begin{minipage}{1in}\includegraphics[width=\textwidth]{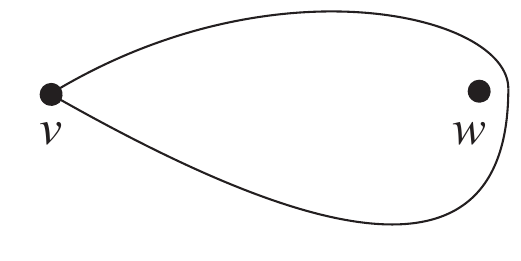}\end{minipage} 
+
\begin{minipage}{1in}\includegraphics[width=\textwidth]{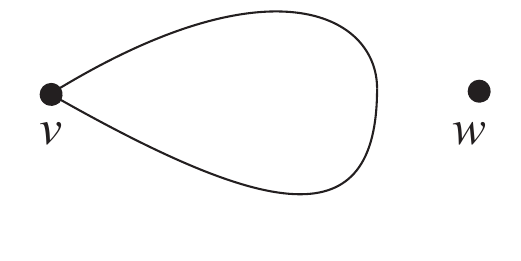}\end{minipage} \right)
=
\left( \begin{minipage}{1in}\includegraphics[width=\textwidth]{arcsquared2.pdf}\end{minipage}  \right)
\]
since any arc bounding an unpunctured monogon in $0$ by the example above. 
\end{proof}

\subsection{Relationship of $\cC(\Sigma_{g, n})$ with hyperbolic geometry}

Let $\cT^{d}(\Sigma_{g, n})$ be the decorated Teichm\"uller space of $\Sigma_{g, n}$ constructed by Penner \cite{Penner87}. It parameterizes all pairs $(m, r)$ where $m$ is a complete hyperbolic metric $\Sigma_{g, n}$ and $r$ is a choice of a horocycle at every puncture of $\Sigma_{g, n}$. Given such a pair $(m, r)$, one can assign a well-defined length to any loop on $\Sigma_{g, n}$ and any arc that goes from puncture to puncture on $\Sigma_{g, n}$. In addition, we set the length of a vertex to be the length of the horocycle around that vertex.  These lengths of loops, arcs, and vertices can then be used  to define $\lambda$-length functions on $\cT^{d}(\Sigma_{g, n})$, and it was shown in \cite{Penner87} that  the $\lambda$-length functions parametrize the ring of $\CC$-valued $C^\infty$ functions on $\cT^{d}(\Sigma_{g, n})$.   These $\lambda$-length functions can be used to define a Poisson structure on $\cT^{d}(\Sigma_{g, n})$ induced by the Weil-Petersson form \cite{Penner92}.

Roger and Yang defined the curve algebra and showed (\cite[Theorem 1.2]{RogerYang14}) that there is a Poisson algebra homomorphism 
\begin{equation}\label{eqn:phi}
	\Phi : \cC(\Sigma_{g, n})_{\CC} \to C^{\infty}(\cT^{d}(\Sigma_{g, n}))
\end{equation}
that  sends any loop, arc, or vertex to its corresponding $\lambda$-length function. One can think of the relations from the curve algebra as designed to mirror the relations from the $\lambda$-length functions in $ C^\infty(\cT^{d}(\Sigma_{g, n}))$. In fact, Roger and Yang conjectured that  the  curve algebra relations captures \emph{all} of the relations from $ C^\infty(\cT^{d}(\Sigma_{g, n}))$, or equivalently, that 

\begin{conjecture}\label{conj:RogerYang}\cite{RogerYang14}
The Poisson algebra homomorphism $\Phi$ in \eqref{eqn:phi} is injective. 
\end{conjecture}

Theorem \ref{thm:defquantTeich} of this paper proves Roger and Yang's conjecture in all cases, by appealing to the algebraic properties of $\cC(\Sigma_{g, n})_{\CC}$ and the following theorem:

\begin{theorem}[\protect{\cite[Theorem A]{MoonWong21}}]\label{thm:IDimpliesRY}
If $\cC(\Sigma_{g, n})_{\CC}$ is an integral domain, then $\Phi$ is injective. 
\end{theorem}

In  previous work \cite[Theorem B and Section 4]{MoonWong21} , we were able to verify that  $\cC(\Sigma)_{\CC}$ is an integral domain when $\Sigma$ admits a `locally planar' ideal triangulation.  In particular, the genus $g$ and the number of punctures $n$ should satisfy
\[
	n \ge \begin{cases}\lceil \frac{7 + \sqrt{48g}}{2}\rceil, & g \ne 2\\
	10, & g = 2. \end{cases}
\]
In Theorem  \ref{thm:intdom} of this paper, we instead use cluster algebras to obtain an  independent and unconditional proof of integrality, so that Conjecture \ref{conj:RogerYang} applies for any $\Sigma_{g, n}$.

\subsection{Relationship of $\cC(\Sigma_{g, n})$ with Kauffman bracket skein algebra}

Roger and Yang's definition of the curve algebra and the construction of $\Phi$ in \cite{RogerYang14} was  motivated by a search for an appropriate quantization of the decorated Teichmuller space $\cT^{d}(\Sigma_{g, n})$.  In particular, they wanted to mimic and generalize the set-up of \cite{Turaev91, Bullock97, BullockFrohmanJKB99, PrzytyckiSikora00} that establishes the Kauffman bracket skein algebra as a quantization of the $\mathrm{SL}_2$-character variety of $\overline{\Sigma}_{g}$, which contains the Teichm\"uller space as a dense open subspace. Towards this goal, Roger and Yang defined a generalized Goldman bracket for $\cC(\Sigma_{g, n})$ and used it to define a deformation quantization that we here denote by $\cS^{q}(\Sigma_{g, n})$ \cite[Theorem 1.1]{RogerYang14}. 

We omit the precise definition of $\cS^{q}(\Sigma_{g. n})$, but instead mention some key properties.  Firstly, $\cS^{q}(\Sigma_{g, n})$ is an $R[q^{\pm \frac{1}{2}}]$-algebra generated by arcs, loops, and vertices, and reduces to the usual Kauffman bracket skein algebra in the absence of punctures (so that the puncture-skein and puncture-framing relations can be ignored).  For this reason, we will refer to Roger-Yang's $\cS^{q}(\Sigma_{g, n})$ as a \emph{skein algebra}. In addition, $\cS^{q}(\Sigma_{g, n})$ can be identified with $\cC(\Sigma_{g, n})$ when $q = 1$.    

As stated in Theorem \ref{thm:defquantTeich}, establishment of Conjecture \ref{conj:RogerYang} implies that  $\cS^{q}(\Sigma_{g, n})$ is indeed a quantization of $\cT^{d}(\Sigma_{g, n})$, completing Roger and Yang's original goal.  Moreover, many of the results about the curve algebra $\cC(\Sigma_{g, n})$ have consequences for the skein algebra $\cS^{q}(\Sigma_{g, n})$.  For example, it was proved in  \cite[Theorem C]{MoonWong21} that if $\cC(\Sigma_{g, n})$ is an integral domain, then $\cS^{q}(\Sigma_{g, n})$ is also an integral domain.

Conversely, many results about $\cS^{q}(\Sigma_{g, n})$ also apply to $\cC(\Sigma_{g, n})$.  The following two theorems about algebraic properties of $\cC(\Sigma_{g, n})$ were proved for  $\cS^{q}(\Sigma_{g, n})$ with $\CC$-coefficients, but the same proof works just as well for $\ZZ$-coefficients and with $q =1$. 

\begin{theorem}[\protect{\cite[Theorem 2.2]{BKPW16JKTR}}]\label{thm:finitegeneration}
The algebra $\cC(\Sigma_{g, n})$ is finitely generated.
\end{theorem}

We now turn to the $g = 0$ case. Let $C$ be a small circle on $\Sigma_{0, n}$.  We may assume that the $n$ punctures $\{v_{1}, \cdots, v_{n}\}$ lie on $C$ in the clockwise circular order. Let $\underline{\beta_{ij}}$ be the simple arc in the disk bounded by $C$ that connects $v_{i}$ and $v_{j}$. 

\begin{theorem}[\protect{\cite[Theorem 1.1]{ACDHM21}}]\label{thm:presentationC0n}
The algebra $\cC(\Sigma_{0, n})$ is isomorphic to 
\[
	\ZZ[v_{i}^{\pm}, \underline{\beta_{ij}}]_{1 \le i, j \le n}/J,
\]
where $J$ is an ideal generated by 
\begin{enumerate}
\item $\underline{\beta_{ik}} \ \underline{\beta_{j\ell}} = \underline{\beta_{i\ell}}\ \underline{\beta_{jk}} + \underline{\beta_{ij}} \ \underline{\beta_{k\ell}}$ 
for any 4-subset $\{i, j, k, \ell\} \subset [n]$ in cyclic order;
\item $\underline{\gamma_{ij}}^{+} = \underline{\gamma_{ij}}^{-}$;
\item $\underline{\delta} = -2$,
\end{enumerate}
where $\underline{\gamma_{ij}}^{\pm}$ and $\underline{\delta}$ are explicit polynomials in the generators (and have a geometric description). 
\end{theorem}
For definitions of $\underline{\gamma_{ij}}^{\pm}$ and $\underline{\delta}$ and formulas in terms of $\underline{\beta_{ij}}$, see \cite[Section 4]{ACDHM21}.

\subsection{A useful variation of $\cC(\Sigma_{g, n})$}

\begin{definition}\label{def:Cprime}
Let $\cC(\Sigma_{g, n})' \subset \cC(\Sigma_{g, n})$ be the subalgebra generated by the following elements:
\begin{enumerate}
\item Isotopy classes of loops;
\item $\underline{\beta}$, \, $v\underline{\beta}$, \, $w\underline{\beta}$, and $vw\underline{\beta}$, where $\underline{\beta}$ is an arc connecting (possibly non-distinct) vertices $v$ and $w$.
\end{enumerate}
\end{definition}

For any coefficient ring $R$, set $\cC(\Sigma_{g, n})_{R}' := \cC(\Sigma_{g, n})' \otimes_{\ZZ}R \subset \cC(\Sigma_{g, n})_{R}$.
Later, we will need a slight extension/variation of Theorem \ref{thm:finitegeneration}. 

\begin{theorem}\label{thm:Cprimefinitegeneration}
For any coefficient ring $R$, the algebra $\cC(\Sigma_{g, n})_{R}'$ is finitely generated. 
\end{theorem}

The proof is identical to that of \cite[Theorem 2.2]{BKPW16JKTR}. More specifically, one uses a generalized handle decomposition of $\Sigma_{g, n}$ with a disk removed. The complexity of a curve is defined based on how many times and in what manner a minimal representation of the curve traverses the handles \cite[Section 3.1]{BKPW16JKTR}. By application of skein identities, it is shown that any curve can be recursively written as lower-complexity curves  \cite[Lemmas 3.1--3.4]{BKPW16JKTR}. Importantly, none of the skein identities in the recursive steps use the formal inverses of vertices. In particular, the skein identities from \cite[Lemmas 3.1 and 3.2]{BKPW16JKTR}) involve only undecorated arcs of the form $\underline{\beta}$, and those for \cite[Lemma 3.3]{BKPW16JKTR} uses arcs of the form $\underline{\beta}$ and $vw \underline{\beta}$. For \cite[Lemma 3.4]{BKPW16JKTR}, one identity (first identity on \cite[p.10]{BKPW16JKTR} ) involves $v^{-1}$.  However, the recursive step comes from substituting it into a previous equation (last identify on \cite[p.9]{BKPW16JKTR}), in a term with a factor of  $v$. Because of the cancellation, the recursive step can be written in a form involving only undecorated arcs.

\begin{remark}
Later, we will see that the cluster algebra $\cA(\Sigma_{g, n})$ (to be defined in Section \ref{sec:clusteralgebrasurface}) can be understood as a subalgebra generated by `tagged' arc classes by Compatibility Lemma. On the other hand, the classical limit  ($q=1$) of the original Kauffman bracket skein algebra  \cite{Przytycki91, Turaev91} is a subalgebra of $\cC(\Sigma_{g, n})$ generated by loop classes. So, one may interpret $\cC(\Sigma_{g, n})'$ as the subalgebra of $\cC(\Sigma_{g, n})$ generated by the image of the cluster algebra $\cA(\Sigma_{g, n})$ and the usual Kauffman bracket skein algebra. 
\end{remark}


\section{Cluster algebra from surfaces}\label{sec:clusteralgebrasurface}

We review the definition of the cluster algebra $\cA(\Sigma_{g, n})$ constructed from a punctured surface $\Sigma$, as introduced by Fomin, Shapiro, and Thurston in \cite{FominShapiroThurston08}. 

\subsection{Definition of cluster algebras}\label{ssec:clusteralgebra}

We begin by noting that we will not need the definition of cluster algebras in full generality, which can be found for example in \cite{FominZelevinsky02}. We will restrict to the case of constant coefficient, skew-symmetric exchange matrix, and  no frozen variables. The only minor extension is that we allow more general base ring including finite field, while in many literature a cluster algebra is defined over $\ZZ$, $\QQ$, $\RR$, or $\CC$. Essentially  the choice of  coefficient ring  does not significantly impact the theory \cite[Section 2]{BMRS15}.

Let $R$ be an integral domain. Let $\cF$ be a purely transcendental finite extension of $Q(R)$, the field of fraction of $R$. A \emph{seed} is a pair $(\bx, B)$, where $\bx = \{x_{1}, \ldots, x_{m}\}$ is a free generating set for $\cF$ as a field over $Q(R)$ and $B=(b_{ij})$ is a skew-symmetric $ m \times m$ integral matrix. $B$ is called the \emph{exchange matrix}, the set $\bx$ is the \emph{cluster}, and its elements $x_{i}$ are the \emph{cluster variables} of the seed.

For a seed $(\bx, B)$ and  $k \in \{1, \ldots, m \}$, a  \emph{mutation  in the direction of $k$} is an operation that produces another seed $\mu_{k} ( \bx, B) = (\bx', B')$ where
\begin{enumerate}
\item  $\bx' = \{ x_{1}', \ldots, x_{m}'\}$ is such that $x_{k}'$ is defined by the \emph{exchange relation}
\[ x_{k} x_{k}' = \prod_{b_{jk}>0} x_{j}^{b_{jk}} + \prod_{b_{jk}<0} x_{j}^{-b_{jk} }\]
and all other cluster variables are identical, so $x_{i} = x_{i}'$ for $i \neq k$;
\item $B' = (b_{ij}')$ is defined by  
\begin{equation}\label{eqn:matrixmutation}
	b_{ij}' = \begin{cases}-b_{ij}, & \mbox{if } i = k \mbox{ or } j = k,\\
	b_{ij} + \frac{1}{2}(|b_{ik}|b_{kj} + b_{ik}|b_{kj}|), & \mbox{otherwise.}\end{cases}
\end{equation}
\end{enumerate}
Sometimes we notate it as $\mu_{k}(B) = B'$.  It is straightforward to check that a mutation is involutive.   

Since a mutation of a seed produces another seed, repeated mutations can be performed following any sequence of indices $1, \ldots, m$.  
We say that two seeds $(\bx, B)$ and $(\by, C)$ are \emph{mutation equivalent} and write $(\bx, B) \sim (\by, C)$ if one seed can be obtained from the other by a sequence of mutations. 

\begin{definition}\label{def:clusteralg}
The \emph{cluster algebra} $\cA(\bx, B)$ is the $R$-subalgebra of the ambient field $\cF$ generated by 
\[
	\bigcup\limits_{(\by, C) \sim (\bx, B)} \by,
\]
the cluster variables of seeds that are mutation equivalent to a seed $(\bx, B)$. Since mutation equivalent seeds produce the same cluster algebra, we write $\cA$ instead of $\cA(\bx, B)$ when the choice of initial seed may be safely suppressed. When we need to specify the coefficient ring, we use the notation $\cA_{R}$ for $\cA$. 
\end{definition}

A simplicial complex, called the \emph{cluster complex} of $\cA = \cA(\bx, B)$, is often used to describe the relationships between the cluster variables used to generate it.  In particular, the vertices of the cluster complex are the cluster variables $\displaystyle \bigcup\limits_{(\by, C) \sim (\bx, B)} \by$ that generate $\cA$, and there is a $k$-simplex whenever  $k$ cluster variables belong to the same cluster.  Thus each seed in a cluster algebra gives rise to a maximal simplex in the cluster complex.   The \emph{exhange graph}  is the dual graph, where the vertices are the seeds, and there is an edge between two seeds if they are mutations of each other.   So by definition, the exchange graph of a cluster algebra must be an $m$-regular, connected graph.  

By the \emph{Laurent phenomenon} \cite[Theorem 3.1]{FominZelevinsky02}, for any $x_{i} \in \bx$ and an equivalent seed $(\by = \{y_{1}, y_{2}, \cdots, y_{m}\}, C) \sim (\bx, B)$, 
\[
	x_{i} \in R[y_{1}^{\pm}, y_{2}^{\pm}, \cdots, y_{m}^{\pm}] \subset \cF.
\]

\begin{definition}\label{def:upperclusteralgebra}
For a cluster algebra $\cA = \cA(\bx, B) \subset \cF$, the \emph{upper cluster algebra} $\cU$ is defined by
\[
	\cU := \bigcap_{(\by, C) \sim (\bx, B)}R[y_{1}^{\pm}, y_{2}^{\pm}, \cdots, y_{m}^{\pm}] \subset \cF.
\]
\end{definition}

The Laurent phenomenon tells that $\cA \subset \cU$. In general they do not coincide. The upper cluster algebra $\cU$ behaves better than $\cA$; for example, $\cU$ is an integrally closed domain if $R$ is \cite[Lemma 2.1]{BMRS15}. However, the computation of $\cU$ and the question of whether $\cA = \cU$ or not are in general difficult. For a partial criterion for $\cA = \cU$, see \cite{Muller13}.

\subsection{Definition of the cluster algebra of a surface}

In this paper, we focus exclusively on cluster algebras associated to a punctured surface $\Sigma_{g, n}$. The cluster algebra $\cA(\Sigma_{g, n})$ is essentially the algebra generated by isotopy classes of arcs on the surface~$\Sigma_{g, n}$.  Each cluster should come from the arc classes in a maximal compatible set; in other words, the edges of a triangulation should form a cluster. A mutation should correspond to a flip of an edge of the triangulation. Although this idea is sufficient to define $\cA(\Sigma_{g, n})$, the intuitive picture is not complete as stands, because not every arc in an ordinary triangulation is flippable if the triangulation contains a self-folded triangle. This problem was resolved  in \cite{FominShapiroThurston08} by introducing tagged arcs.  

We begin with a review of ordinary triangulations, and how the data from a single triangulation without a self-folded triangle is sufficient to define a cluster algebra $\cA(\Sigma_{g, n})$.  We then introduce tagged triangulations, which will fully describe the correspondence between cluster variables and tagged arcs. The results in \cite{FominShapiroThurston08} also apply to surfaces with boundary and marked points on the boundary, but we do not need that generality here. 

\subsubsection{Ordinary Triangulations}\label{ssec:notation}

As in Section \ref{sec:arcalgebra}, we denote a punctured surface without boundary by $\Sigma_{g, n} = \overline{\Sigma}_{g} \setminus V$, where $V = \{v_{1}, \cdots, v_{n}\}$. We assume that $n \ge 1$, and exclude $\Sigma_{0, n}$ for $n \le 3$. 

Recall an arc of $\Sigma_{g, n} = \overline{\Sigma}_{g} \setminus V$ is an immersion $\underline{\alpha} : [0, 1] \to \overline{\Sigma}_{g}$ such that $\underline{\alpha}$ embeds $(0, 1)$ in $\Sigma_{g, n}$ and $\underline{\alpha}$ takes the endpoints $\{0, 1\}$ to the punctures $V$. The set of isotopy classes of arcs connecting two punctures in $\Sigma_{g, n}$ will be denoted by $\bA^{\circ}(\Sigma_{g, n})$. Two arcs are said to be \emph{compatible} if they are the same, or if they do not intersect except at the punctures.  A maximal collection of distinct, pairwise compatible isotopy classes of arcs forms an \emph{ideal triangulation} $ \cT$ on $\Sigma_{g, n}$.  The arcs in a triangulation are referred to as \emph{edges}, and the set of edges is denoted  by $E$.  Because of maximality,  $E$ separates $\Sigma_{g, n}$ into a set of triangles, which is denoted by $T$.  Recall that $n = |V|$, and from now on, we let $m = |E|$.

A \emph{flip} is an operation that removes an arc from a triangulation $\cT$, replaces it by another compatible arc, so results in another triangulation $\cT'$.  So  $\cT$ and $\cT'$ share all arcs except one.   Note that not every arc in a triangulation is flippable; in particular, the folded edge in a self-folded triangle is not flippable.  However, there is a finite sequence of flips that transforms any triangulation into one without self-folded triangles, and more generally,  any two triangulations can be connected by finitely many flips.    

Let the \emph{arc complex} $\Delta^{\circ}(\Sigma_{g, n})$ be the abstract simplicial complex where a $k$-simplex is a collection of $k$ distinct, mutually compatible arcs in $\bA^{\circ}(\Sigma_{g, n})$.  Thus each vertex is an isotopy class of an arc, and a maximal simplex corresponds to a triangulation $\cT$. Its \emph{dual graph} we denote by $\bE^{\circ}(\Sigma_{g, n})$. Equivalently,  $\bE^{\circ}(\Sigma_{g, n})$ is the graph whose vertices are the ideal triangulations of $\Sigma_{g, n}$ and two vertices are connected if and only if the ideal triangulations are related by a flip. $\Delta^{\circ}(\Sigma_{g, n})$ is connected in codimension-one, and $\bE^{\circ}(\Sigma_{g, n})$ is connected, with each vertex degree at most $m$.  

\subsubsection{Cluster algebra from an ordinary triangulation}
The combinatorial data from an ordinary triangulation can be encoded using a matrix, which we will define using puzzle pieces. Figure \ref{fig:puzzle} shows three ``puzzle pieces'' which are intended to be glued together along their boundary edges in order to construct triangulations of surfaces.  Figure \ref{fig:puzzle4} depicts a triangulation of the four-punctured sphere  $\Sigma_{0, 4}$, where the exterior of three self-folded triangles is another triangle, which is not drawn in but which should be understood to be a part of the figure.   We sometimes refer to the triangulation in Figure \ref{fig:puzzle4} as a fourth puzzle piece, even though it is not meant to be glued to any other puzzle piece.     The matrix associated to the puzzle pieces are also given in  Figures  \ref{fig:puzzle} and  \ref{fig:puzzle4}.   Notice that there is one row and column for each edge in the puzzle piece, and all four matrices are skew-symmetric.

As shown in \cite[Section 4]{FominShapiroThurston08}, every triangulation $\cT$ of $\Sigma_{g, n}$ can be obtained from gluing puzzle pieces of the four types depicted in Figures \ref{fig:puzzle} and \ref{fig:puzzle4}.   Moreover, there is a well-defined \emph{exchange matrix}  $B = B_{\cT} = (b_{ij})$ that is the $m \times m$ matrix whose rows and columns are indexed by the edges of the triangulation, constructed as the sum of all minor matrices obtained from some set of puzzle pieces which can be used to construct $\cT$. Since an edge of a triangulation can be contained in at most two puzzle pieces, the entries of the exchange matrix must satisfy $-2 \le b_{ij} \le 2$ for all $i, j$.  We refer the reader to  \cite{FominShapiroThurston08} for details as well as worked examples.  

\begin{figure}
\begin{tabular}{ccc}
\includegraphics[height=1in]{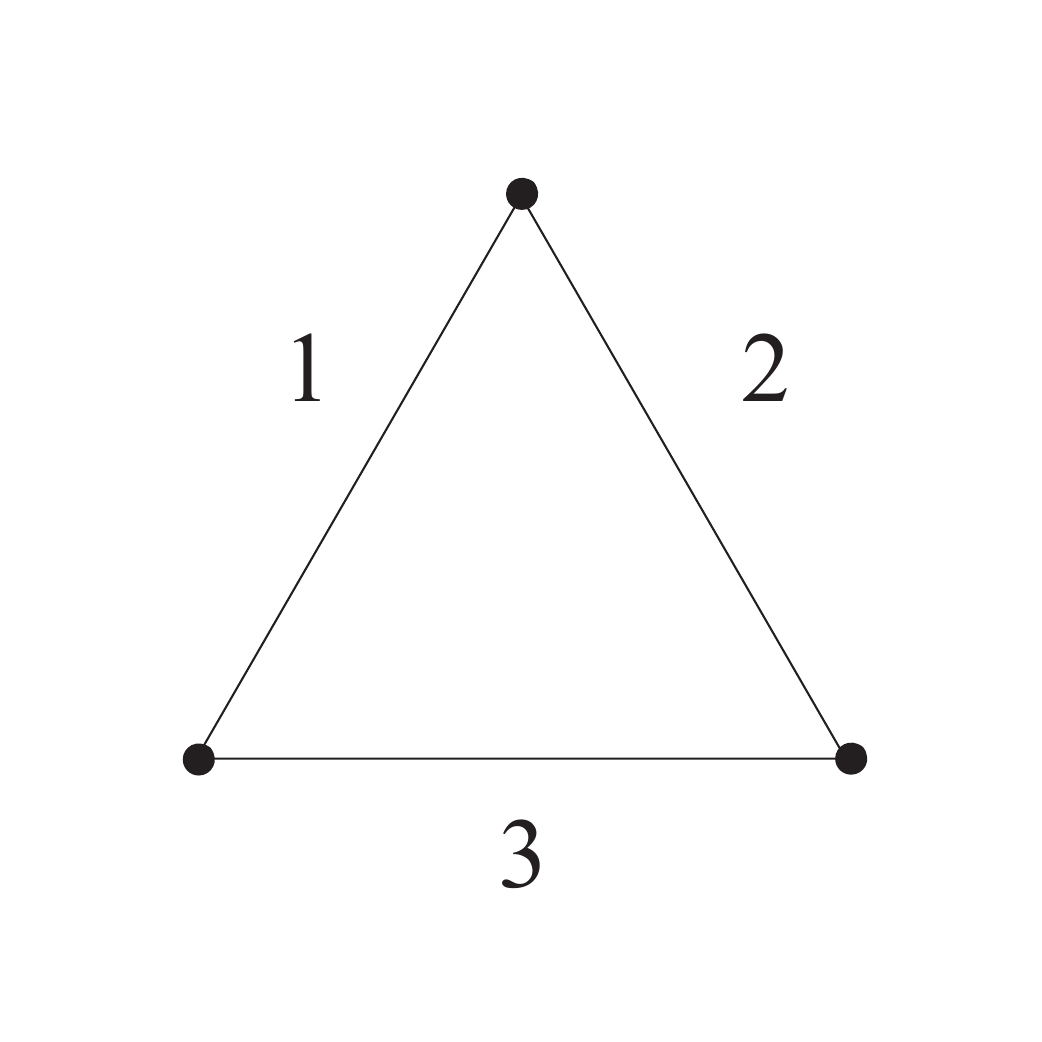}&
\includegraphics[height=1in]{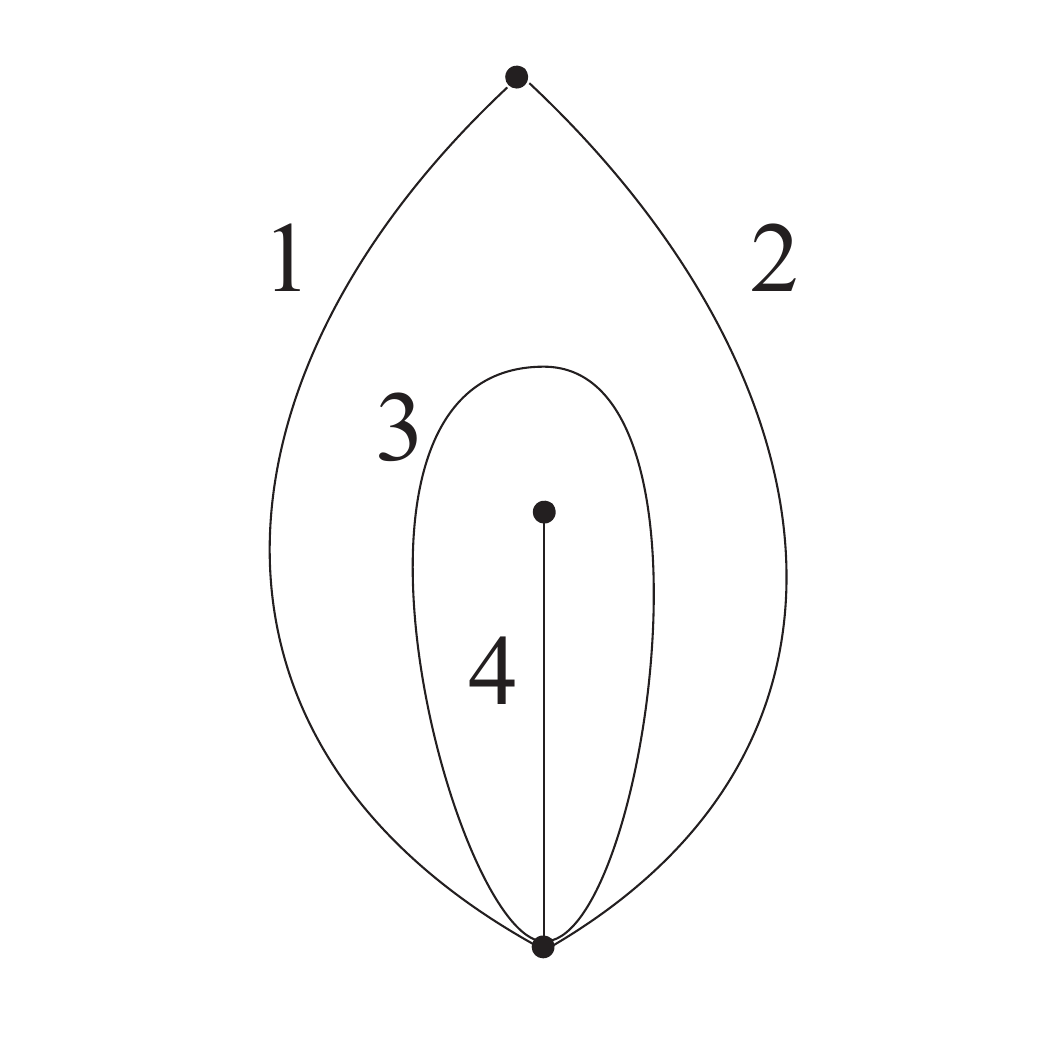}&
\includegraphics[height=1in]{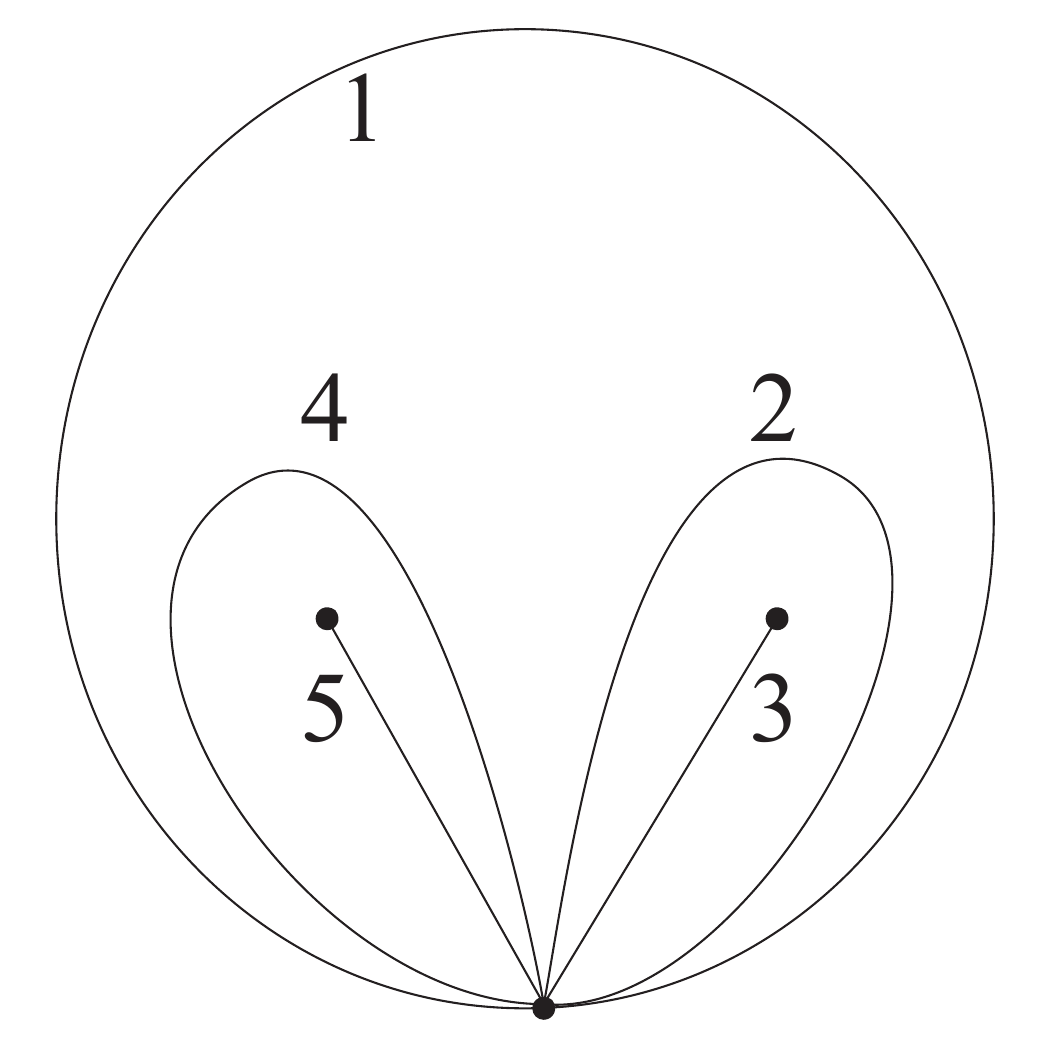}\\
$\left[\arraycolsep=2pt\begin{array}{rrr}  0&1&-1\\-1&0&1\\1&-1&0 \end{array}\right]$&
$\left[\arraycolsep=2pt\begin{array}{rrrr}0&1&-1&-1\\-1&0&1&1\\1&-1&0&0\\1&-1&0&0\end{array}\right] $&
$\left[\arraycolsep=2pt\begin{array}{rrrrr}0&1&1&-1&-1\\-1&0&0&1&1\\-1&0&0&1&1\\1&-1&-1&0&0\\1&-1&-1&0&0\end{array}\right]$
\end{tabular}
\caption{Three puzzle pieces and their associated matrix minors}
\label{fig:puzzle}
\end{figure}

\begin{figure}
\begin{tabular}{cc}
\begin{minipage}{1.3in}\includegraphics[height=1.3in]{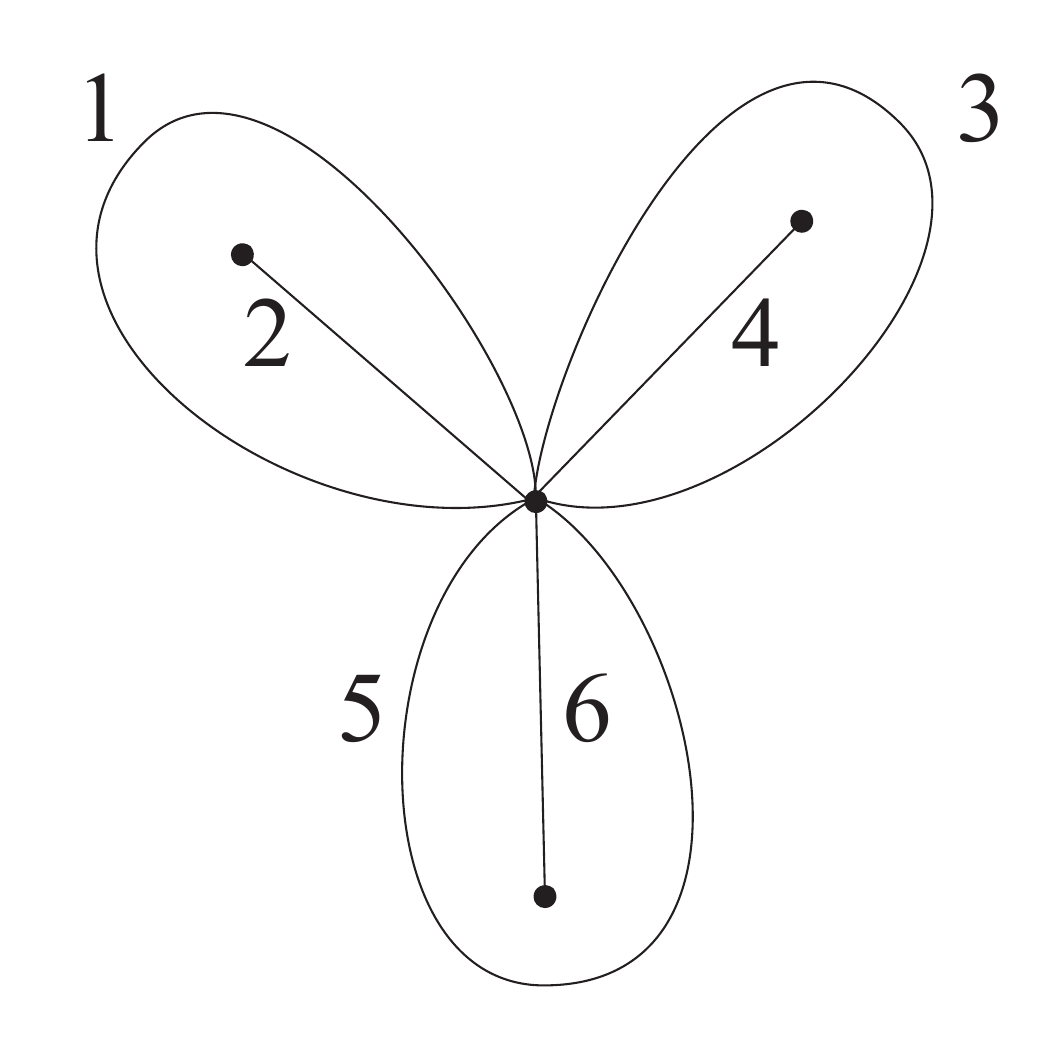}\end{minipage} & 
$\arraycolsep=2pt\left[\begin{array}{rrrrrr} 0&0&-1&-1&1&1\\0&0&-1&-1&1&1\\1&1&0&0&-1&-1\\1&1&0&0&-1&-1\\-1&-1&1&1&0&0\\-1&-1&1&1&0&0 \end{array}\right]$
\end{tabular}
\caption{The fourth puzzle pieces and their associated matrix minors}
\label{fig:puzzle4}
\end{figure}
 
Observe that the exchange matrix $B_{\cT}$ is skew-symmetric, since the minor matrices obtained from the puzzle pieces are skew-symmetric.  Thus, we may define the \emph{seed from the triangulation $\cT$} to be the pair $(E_{\cT}, B_{\cT})$, where $E_{\cT}$ is the set of edges of a triangulation $\cT$ and $B$ is its exchange matrix.  

\begin{proposition}[\protect{\cite[Proposition 4.8]{FominShapiroThurston08}}] \label{prop:flip}
Suppose that the $k$-th edge of a triangulation  $\cT$ is flippable, and let $\cT'$ be the result of flipping that edge.  Then the exchange matrix for $\cT'$ is the exchange matrix for $\cT$ mutated in the direction $k$, i.e.,  $B_{\cT'} = \mu_{k}(B_{\cT})$. 
\end{proposition}

Since any two triangulations of $\Sigma_{g, n}$ are related by a sequence of flips, seeds from any two triangulations of $\Sigma_{g, n}$ are related by a sequence of mutations and hence are mutation equivalent.  Hence, we have:

\begin{definition}
Let the \emph{cluster algebra of $\Sigma_{g, n}$} be defined as $\cA(\Sigma_{g, n})= \cA(E_{\cT}, B_{\cT})$. Then $\cA(\Sigma_{g, n})$ is generated by the edges of triangulations of $\Sigma_{g, n}$ and hence is independent of the initial choice of triangulation $\cT$.
\end{definition}

However, the arcs are insufficient to describe all cluster variables. In a cluster algebra, we necessarily are able to mutate along every edge of a triangulation, but when the surface $\Sigma_{g, n}$ admits a triangulation with self-folded triangles, not every edge is flippable. In other words, for some vertex in $\bE^{\circ}(\Sigma_{g, n})$, the degree might be strictly smaller than $m$, while the exchange graph of $\cA(\Sigma_{g, n})$ has to be $m$-regular. So, we can only in general say that $\Delta^{\circ}(\Sigma_{g, n})$ is a subcomplex of the cluster complex, and  $\bE^{\circ}(\Sigma_{g, n})$ is a subgraph of the cluster algebra's exchange graph. To fill in this gap,  Fomin, Shapiro, and Thurston \cite{FominShapiroThurston08} introduced a generalization of ordinary arcs which we describe next.

\subsubsection{Tagged Triangulations}\label{ssec:taggedtriangulations}

\begin{definition}\label{def:taggedarc}
 A \emph{tagged arc} $\alpha$ on $\Sigma_{g, n}$ is an arc $\underline{\alpha}$ on $\Sigma_{g, n}$ along with one of two decorations, \emph{plain} or \emph{notched}, at each of the two ends of $\underline{\alpha}$ such that:
\begin{enumerate}
\item $\underline{\alpha}$ does not cut a one-punctured monogon; 
\item if both ends of the arc are at the same vertex, then they have the same decoration. 
\end{enumerate}
The ordinary arc $\underline{\alpha}$ is the \emph{underlying arc} of the tagged arc $\alpha$. The decoration of plain or notched at an end of a tagged arc is referred to as the \emph{tag} at that end, or at the corresponding vertex. The set of isotopy classes of tagged arcs is denoted by $\bA^{\bowtie}(\Sigma_{g, n})$. Naturally $\bA^{\circ}(\Sigma_{g, n}) \subset \bA^{\bowtie}(\Sigma_{g, n})$. 
\end{definition}

Many concepts and constructions for arcs can be extended to tagged arcs. Recall that two ordinary arcs are compatible if, up to isotopy, they are either the same or disjoint except at the vertices.  

\begin{definition}\label{def:taggedcompatibility}
 If tagged arcs $\alpha$ and $\beta$ satisfy the following conditions:
\begin{enumerate}
\item the underlying arcs are $\underline{\alpha}$ and $\underline{\beta}$ are compatible; and
\item in the case that $\underline{\alpha} = \underline{\beta}$, then $\alpha$ and $\beta$ have the same tag on at least one of the shared vertices; 
\item in the case that $\underline{\alpha} \ne \underline{\beta}$ and they share a vertex $v$, then $\alpha$ and $\beta$ have the same tag at~$v$.
\end{enumerate}
then we say that $\alpha$ and $\beta$ are \emph{compatible}. 
\end{definition}

It follows from the definition that, if $\alpha$ and $\beta$ are compatible tagged arcs whose underlying arcs are not the same but share both vertices, then $\alpha$ and $\beta$ must have the same tag at each vertex.   For example, on a one-punctured surface, all compatible arcs share a vertex, and hence all ends of compatible arcs must have the same tag.

\begin{definition}\label{def:taggedtriangulation}
A \emph{tagged triangulation} $\cT^{\bowtie}$ is a maximal collection of compatible, distinct tagged arcs.  
\end{definition}

If we take the arcs of an ordinary triangulation $\cT$ and tag all of the ends plainly, then we obtain a tagged triangulation. However, the converse is not true; it is possible that the underlying curves of a tagged triangulation $\cT^{\bowtie} = \{\alpha_{i}\}$ do not form an ordinary triangulation of $\Sigma_{g, n}$. In particular, tagged triangulations may cut out bigons as pictured on the right of Figure \ref{fig:dangle}.  Because such bigons appear often in tagged triangulations, we have the following language for describing them.   

\begin{definition}\label{def:dangle}
Let $v$ and $w$ be two distinct vertices. A \emph{dangle} $d_{v}^{w}$ is a bigon with vertices at $v$ and $w$ such that its two boundary arcs are compatible and have different tags at the vertex~$v$ (Figure \ref{fig:dangle}). The \emph{jewel} of $d_{v}^{w}$ is the vertex $v$ with two distinct tags. An \emph{envelope of the dangle~$d_{v}^{w}$} is the boundary $\gamma_{v}^{w}$ of a  one-punctured monogon that is based at $w$ and such that it encloses the jewel $v$ and has the same tags at $w$ as on $d_{v}^{w}$. 
\end{definition}

\begin{figure}[!ht]
\includegraphics[height=0.1\textheight]{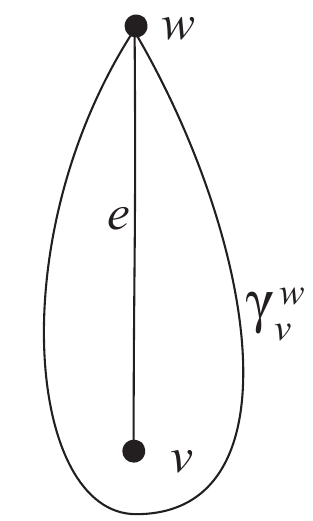}
\qquad
\qquad
\includegraphics[height=0.1\textheight]{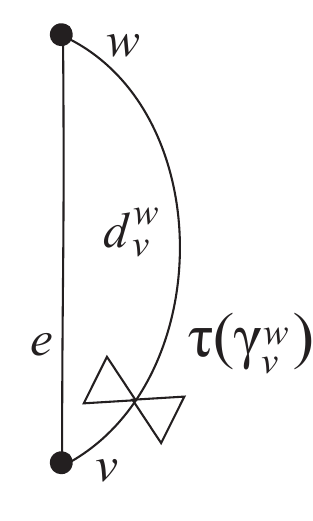}
\caption{On the left, the envelope $\gamma_{v}^{w}$ encircles its jewel $v$.  On the right is the corresponding dangle $d_{v}^{w}$, with the taggings necessarily distinct at $v$.  In this example, both the tags are plain at $w$, but both could be notched at $w$ instead.}
\label{fig:dangle}
\end{figure}

Note that, because the two boundary arcs of a dangle $d_{v}^{w}$ are compatible and the tags at the jewel $v$ are different, the tags at the remaining vertex $w$ must be both plain or both notched. In a tagged triangulation, the jewel of a dangle cannot be the endpoint of any other edge besides those of the dangle, and thus the degree of the jewel is two.

Let the \emph{tagged arc complex} $\Delta^{\bowtie}(\Sigma_{g, n})$ be the abstract simplicial complex generated by compatible distinct tagged arcs in $\bA^{\bowtie}(\Sigma_{g, n})$, and let $\bE^{\bowtie}(\Sigma_{g, n})$ be the dual graph of $\Delta^{\bowtie}(\Sigma_{g, n})$. Equivalently, $\bE^{\bowtie}(\Sigma_{g, n})$ is the graph whose vertices are the tagged triangulations of $\Sigma_{g, n}$ and two vertices are connected if and only if the tagged triangulations share all but one edge. An edge of $\bE^{\bowtie}(\Sigma_{g, n})$ corresponds to a  \emph{tagged flip}, which we think of as an operation that removes one tagged arc from the tagged triangulation and replaces it with a different compatible tagged arc.   

\begin{proposition}[\protect{\cite[Proposition 7.10]{FominShapiroThurston08}}]
Let $m$ be the number of edges of an ideal triangulation on $\Sigma_{g, n}$.  

When $n \geq 2$, $\bE^{\bowtie}(\Sigma_{g, n})$ is an $m$-regular, connected graph. Every edge of a tagged triangulation is flippable and any two tagged triangulations is related by a sequence of tagged flips. 

When $n=1$, $\bE^{\bowtie}(\Sigma_{g, n})$ is an $m$-regular graph with two isomorphic connected components, one where all tags are plain and one where all tags are notched.  
\end{proposition}

It follows that $\Delta^{\bowtie}(\Sigma_{g, n})$ is also connected when there are at least two punctures and has two isomorphic connected components when there is exactly one puncture.  Note that in the case of one puncture, each connected component of $\bE^{\bowtie}(\Sigma_{g, 1})$  is isomorphic to $\bE^{\circ}(\Sigma_{g, 1})$, and each component of the tagged arc complex $\Delta^{\bowtie}(\Sigma_{g, 1})$ is isomorphic to  $\Delta^{\circ}(\Sigma_{g, 1})$.  For simplicity, we will restrict to the component where all tags are plain in the one puncture case for ease of exposition.  With this convention, we have that both $\bE^{\bowtie}(\Sigma_{g, n})$ and $\Delta^{\bowtie}(\Sigma_{g, n})$ are connected in all cases. 

The relationship between the ordinary set-up and the tagged one can be described by a map $\tau : \bA^{\circ}(\Sigma_{g, n}) \to \bA^{\bowtie}(\Sigma_{g, n})$, which we will define using the language of dangles and envelopes from Definition \ref{def:dangle} and Figure~\ref{fig:dangle}. If $e \in \bA^{\circ}(\Sigma_{g, n})$ is not an envelope (that is, it does not cut out a once-punctured monogon), then $\tau(e)$ is $e$ tagged plain at both ends. If $e$ is an  envelope based at $w$ and surrounding $v$, then $\tau(e)$ is the unique arc enclosed by $e$ that connects $v$ and $w$ and that is notched at $v$.  For example, in Figure \ref{fig:dangle}, $\tau(e) = e$, but $\tau$ maps the envelope $\gamma_v^{w}$ to the tagged arc on the right.  
 
  As shown in \cite[Section 7]{FominShapiroThurston08}, $\tau$ preserves the compatibility of arcs and provides a way of mapping an ordinary triangulation to a tagged triangulation.  In this way, we can understand
$\Delta^{\circ}(\Sigma_{g, n})$ as a subcomplex of $\Delta^{\bowtie}(\Sigma_{g, n})$ (though possibly it is not an induced subcomplex), and
$\bE^{\circ}(\Sigma_{g, n})$ as a subgraph of $\bE^{\bowtie}(\Sigma_{g, n})$. 

To define the exchange matrix of a tagged triangulation, we again use puzzle pieces, as drawn in Figure ~\ref{fig:taggedpuzzle}.  As before, the fourth puzzle piece by itself is a tagged triangulation of the four-punctured sphere $\Sigma_{0, 4}$. Since it does not have any exterior edge, it cannot be glued with any other puzzle pieces. 

\begin{figure}
\begin{tabular}{cccc}
\includegraphics[height=1.3in]{puzzlepiece1}&
\includegraphics[height=1.3in]{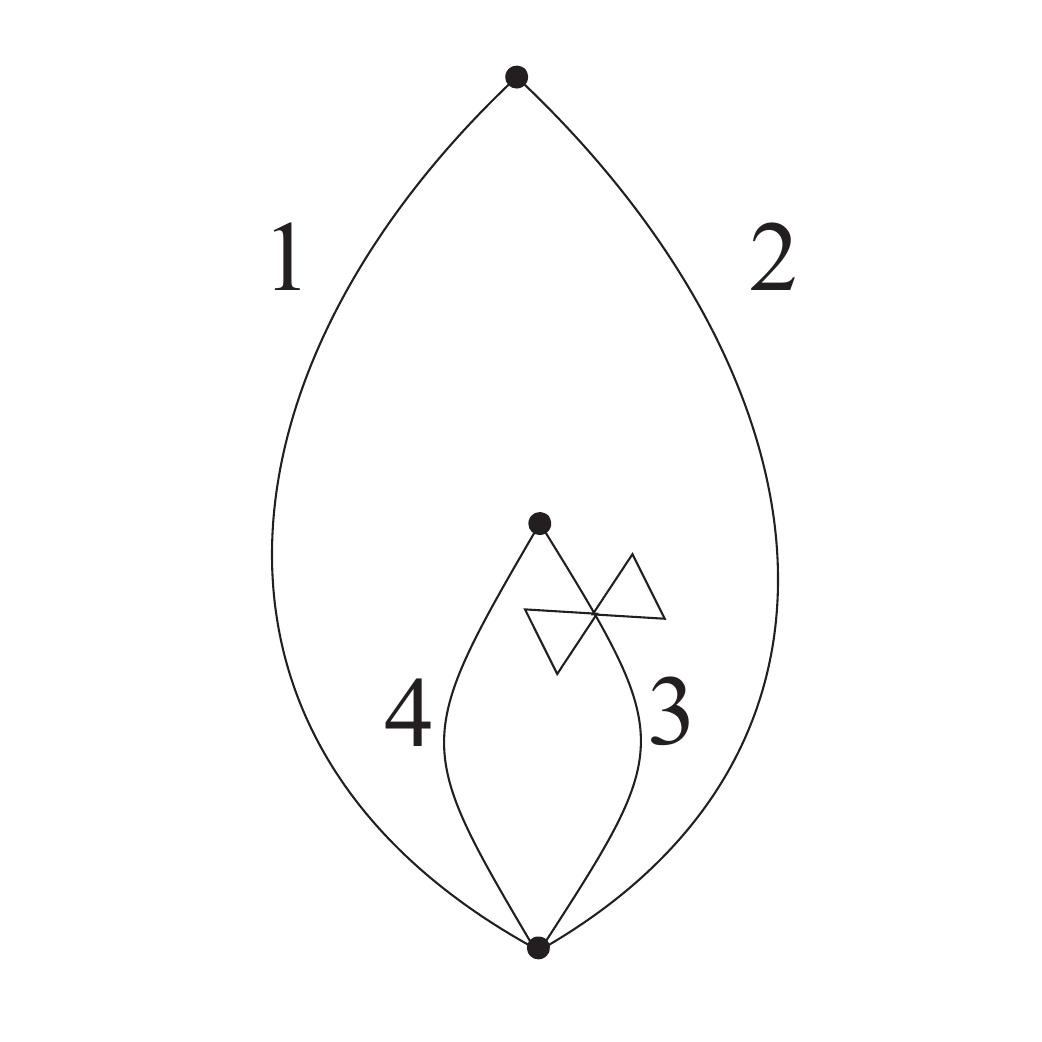}&
\includegraphics[height=1.3in]{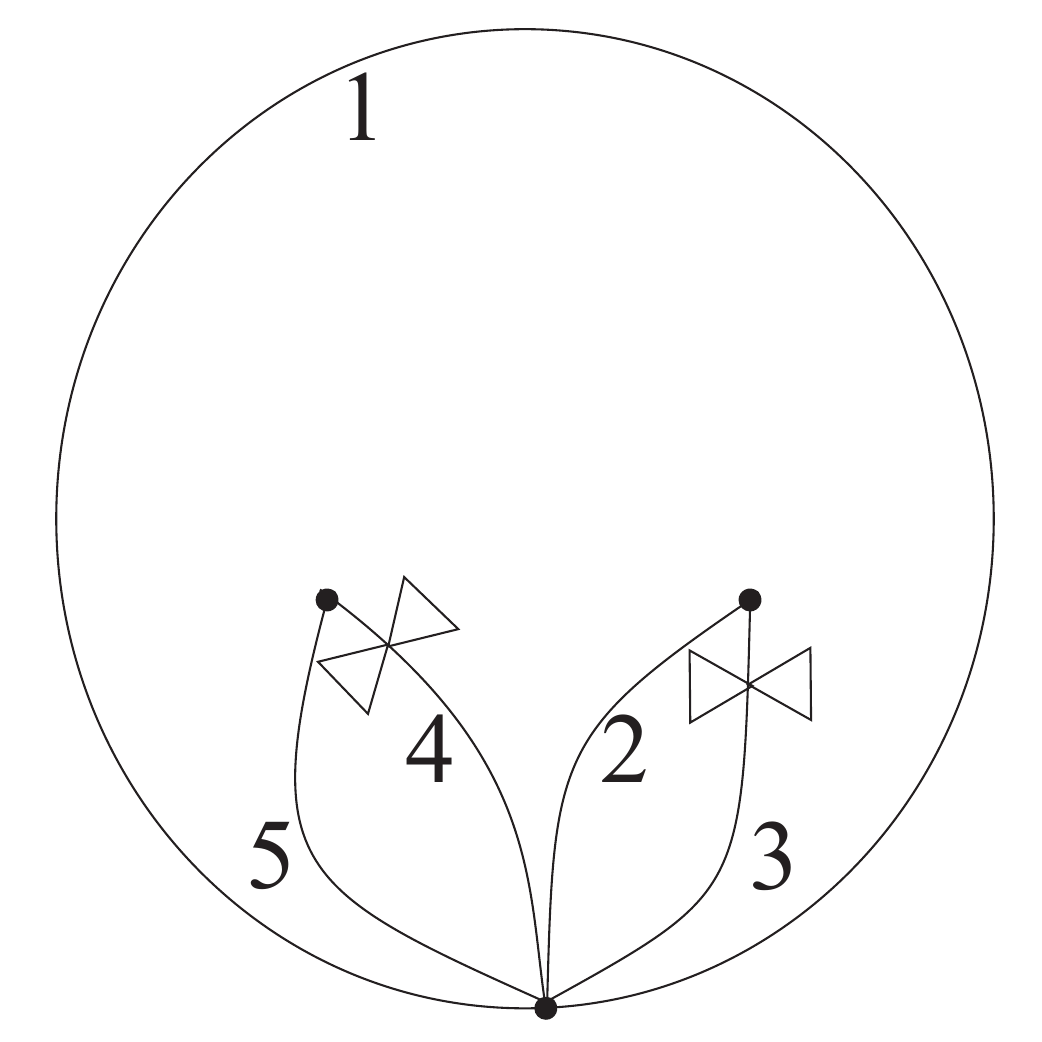}&
\includegraphics[height=1.3in]{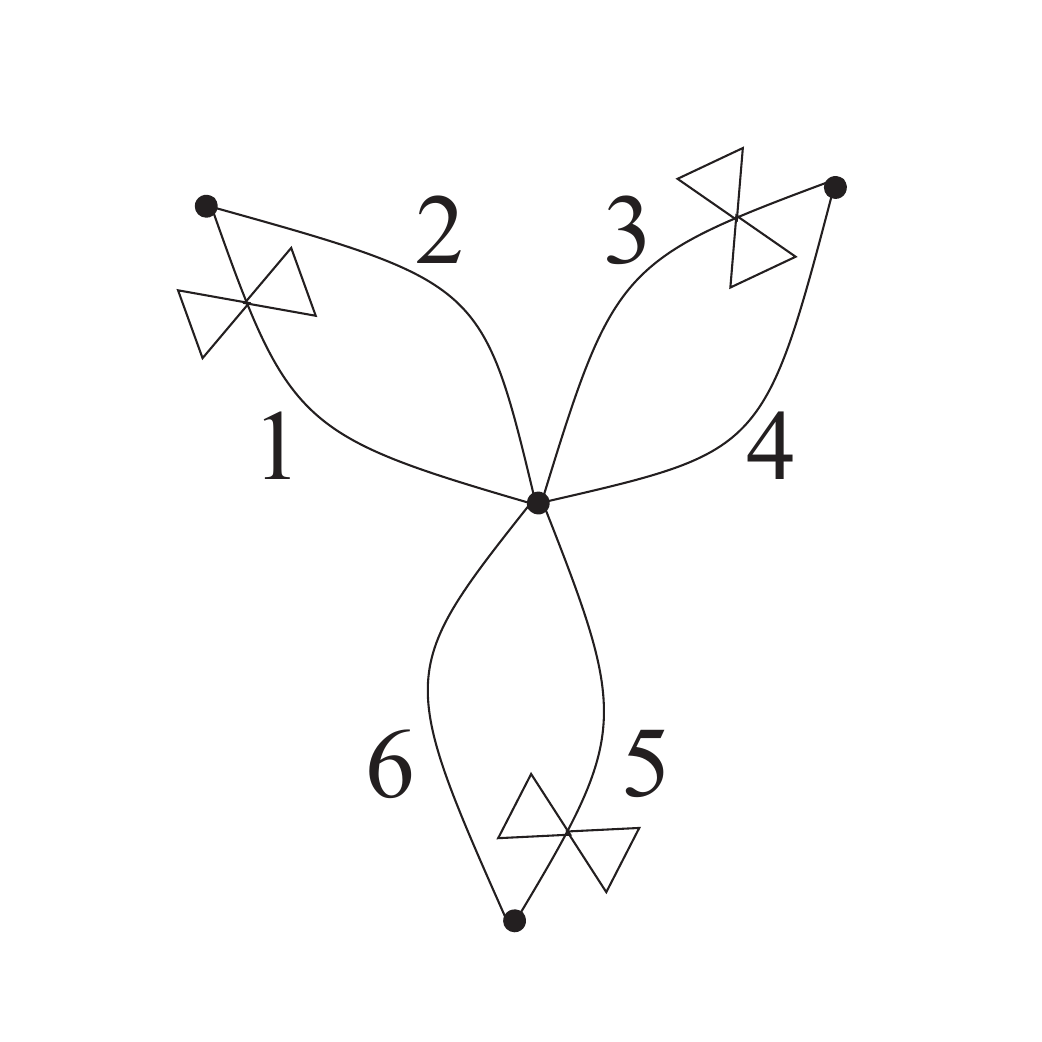}\\
A & B & C & D
\end{tabular}
\caption{The four tagged puzzle pieces.  They are the images under $\tau$ of the four ordinary puzzle pieces from Figures~\ref{fig:puzzle} and \ref{fig:puzzle4}.}
\label{fig:taggedpuzzle}
\end{figure}

\begin{lemma}\label{lem:taggedpuzzlepieces}
Any tagged triangulation $\cT^{\bowtie}$ on $\Sigma_{g, n}$ is obtained by 
\begin{enumerate}
\item gluing the tagged puzzle pieces along their boundary edges; and 
\item tagging all ends of the glued boundary edges in a compatible way. 
\end{enumerate}
\end{lemma}

\begin{proof}
For the given tagged triangulation $\cT^{\bowtie}$, we may think it as a top dimensional simplex in $\Delta^{\bowtie}(\Sigma_{g, n})$. Take a subsimplex $\cS^{\bowtie} \subset \cT^{\bowtie}$, by eliminating all dangles. Then at each vertex of $\cS^{\bowtie}$, the adjacent tagged arcs have the same tag. 

Pick a region $R \subset \Sigma_{g, n}$ bounded by arcs in $\cS^{\bowtie}$. It is sufficient to show that $R$ is one of the tagged puzzle pieces. $R$ is bounded by at most three arcs. Otherwise we can refine the triangulation $\cT^{\bowtie}$ by introducing a new tagged arc dividing the region $R$, which violates the maximality of $\cT^{\bowtie}$.  There is no inner vertex $v$ except the other end of dangles, because otherwise we can insert another compatible tagged edge connecting $v$ and one of the boundary vertices. If $R$ has $k \le 3$ boundary arcs, then there are $3 - k$ dangles in $R$, by the maximality of $\cT^{\bowtie}$. Then Figure \ref{fig:taggedpuzzle} are the remaining possibilities. 
\end{proof}

Observe that the four tagged puzzle pieces in Figure ~\ref{fig:taggedpuzzle} are the images of the four ordinary puzzle pieces in Figures~\ref{fig:puzzle} and ~\ref{fig:puzzle4} under the map $\tau$.  We define the matrix associated to each tagged puzzle piece as the same one associated to its corresponding ordinary puzzle piece.  Note when two tagged arcs have the same underlying arc, their corresponding matrix entries are the same.

\begin{definition}\label{def:exchangematrix}
Let $\cT^{\bowtie}$ be a tagged triangulation with $m$ edges that is made up of tagged puzzle pieces, and let $E_{\cT^{\bowtie}}$ be the set of its edges.   The \emph{exchange matrix} $ B_{\cT^{\bowtie}} = (b_{ij})$ is the $m \times m$ matrix whose rows and columns are indexed by the edges, constructed as the sum of all minor matrices obtained from the puzzle pieces used to construct $\cT^{\bowtie}$.  The \emph{seed} from the triangulation $\cT^{\bowtie}$ is the pair $(E_{\cT^{\bowtie}}, B_{\cT^{\bowtie}})$. 
\end{definition}

\begin{example}\label{ex:exchangematrixtaggedtriangulation}
Consider the tagged triangulation of $\Sigma_{0, 4}$ shown on the left of Figure \ref{fig:mutation7}.  It is obtained from gluing together two puzzle pieces of Type B. Taking $e_{6} = \alpha$, the exchange matrix for the triangulation on the left is 
\[  \arraycolsep=2pt \left[\begin{array}{rrrrrr}0&0&0&0&1& -1\\0&0&0&0&1& -1\\0&0&0&0&-1& 1\\0&0&0&0&-1& 1\\-1&-1&1&1&0&0\\ 1&1&-1&-1&0 &0\end{array}\right]. \]
Mutation of $\alpha$ produces the triangulation on the right of Figure \ref{fig:mutation7}, which is by itself the Type D puzzle piece. The mutated exchange matrix $\mu(B)$ is the one from Figure \ref{fig:puzzle4}.
\end{example}

It is a straightforward calculation to check that the exchange matrix for the tagged triangulation obtained from flipping the $k$-th edge of $\cT^{\bowtie}$ is the exchange matrix for $\cT^{\bowtie}$ mutated in the direction $k$. 

The following theorem, which is the main result of \cite{FominShapiroThurston08}, summarizes our discussion so far. In the case $|V| = 1$, recall that we restricted to the case where all tags are plain, so that $\bE^{\bowtie}(\Sigma_{g, 1})$ is an $m$-regular, connected graph in all cases.  

\begin{theorem}[\protect{\cite[Theorem 7.11]{FominShapiroThurston08}}]\label{thm:clusteralgebraofsurface}
Define the cluster algebra $\cA(\Sigma_{g, n})$ using an initial seed coming from any ordinary or tagged triangulation of $\Sigma_{g, n}$. Then each seed of $\cA(\Sigma_{g, n})$ comes from a tagged triangulation of $\Sigma_{g, n}$, and mutation of the seed corresponds to tagged flips of the triangulation. In particular, the cluster complex of $\cA(\Sigma_{g, n})$ is the tagged arc complex $\Delta^{\bowtie}(\Sigma_{g, n})$ and the exchange graph of $\cA(\Sigma_{g, n})$ is the dual graph $\bE^{\bowtie}(\Sigma_{g, n})$ . 
\end{theorem}

\begin{remark}\label{rmk:quantumclusteralgebra}
As one can see in Figure \ref{fig:puzzle4} or Example \ref{ex:exchangematrixtaggedtriangulation}, the exchange matrix for a punctured surface is not of full rank. Thus, in contrast to the case of surface with boundaries and without punctures, $\cA(\Sigma_{g, n})$ does not admit a quantum cluster algebra as its deformation quantization \cite[Proposition 3.3]{BerensteinZelevinsky05}.
\end{remark}


\section{The homomorphism $\rho : \cA(\Sigma_{g, n}) \to \cC(\Sigma_{g, n})$ }\label{sec:rho}

In this section, we prove Compatibility Lemma in Section \ref{ssec:compatibility}, that there is a monomorphism  $\rho: \cA(\Sigma_{g, n}) \to \cC(\Sigma_{g, n})$.  After describing $\rho$, we prove in Proposition~\ref{prop:welldef} that it is a well-defined algebra homomorphism, and in Proposition~\ref{prop:inj}  that it is injective. 

\begin{definition}\label{def:rholocal}
Let $\alpha \in \cA(\Sigma_{g, n})$ be a tagged arc with endpoints at the vertices $v, w \in V$ (which are possibly the same). Let
\[
	\rho(\alpha) := \begin{cases}\underline{\alpha}, & \mbox{if both ends of $\alpha$ are plain}\\
	v\underline{\alpha}, & \mbox{if only the end at $v$ of $\alpha$ is notched}\\
	w\underline{\alpha}, & \mbox{if only the end at $w$ of $\alpha$ is notched}\\
	vw\underline{\alpha}, & \mbox{if both ends of $\alpha$ are notched}.\end{cases}
\]
where $\underline{\alpha}$ denotes the underlying arc (Definition \ref{def:taggedarc}). 
\end{definition}

\begin{remark}
\begin{enumerate}
\item When $v = w$, both ends of $\alpha$ must have the same decoration (Definition \ref{def:taggedarc}). So the formula is $\rho(\alpha) = \underline{\alpha}$  if both ends are plain, and $\rho(\alpha) = v^{2}\underline{\alpha}$ if both ends are notched.   
\item When there is only one puncture, all endpoints of arcs are tagged plainly. So $\rho(\alpha) = \underline{\alpha}$ for edges $\alpha$ in a once-punctured surface. 
\item   For a related perspective for the definition of $\rho$, see \cite[Lemma 10.14]{FominThurston18}.  
\end{enumerate}
\end{remark}

By introducing a little more notation, we can write the formula for $\rho$ more compactly. For a tagged arc $\alpha$ with an endpoint at $v \in V$, let 
\[
	t_{v}(\alpha) := \begin{cases}0 & \mbox{if $\alpha$ is decorated plainly at $v$},\\
	1 & \mbox{if $\alpha$ is decorated notched at $v$}. \end{cases}
\]
Then Definition \ref{def:rholocal} becomes
\[
	\rho(\alpha) := v^{t_{v}(\alpha)}w^{t_{w}(\alpha)}\underline{\alpha}.
\]
for an edge $\alpha$  whose endpoints are $v$ and $w$.  

\begin{proposition} \label{prop:welldef}
There is a well-defined algebra homomorphism  $\rho : \cA(\Sigma_{g, n}) \to \cC(\Sigma_{g, n})$ that extends Definition \ref{def:rholocal}. 
\end{proposition}

\begin{proof}
Recall that $\cA(\Sigma_{g, n})$ is generated by the edges of all tagged triangulations of $\Sigma_{g, n}$, subject to the exchange relations determined by the mutations.  $\rho$ is already defined for all edges of tagged triangulations, and we can extend it uniquely to the polynomial subalgebra of $\cF$ freely generated by the edges of all tagged triangulations of $\Sigma_{g, n}$. We need to show this map preserves the exchange relations coming from tagged flips along any edge of any tagged triangulation.

With that goal in mind, let $\alpha$ be an arbitrary edge of an arbitrary tagged triangulation $\cT^{\bowtie}$.  
Let the ends of $\alpha$ be $v$ and $w$ (which are possibly the same).  By Lemma~\ref{lem:taggedpuzzlepieces}, we may assume that $\cT^{\bowtie}$ was constructed using tagged puzzle pieces. We split our proof into parts: when $\alpha$ is in a dangle and when it is not.  

{\bf Step 1.} Assume that $\alpha$ is not in a dangle. Then $\alpha$ must be an edge shared by two tagged puzzle pieces of type A, B, or C as depicted in Figure~\ref{fig:taggedpuzzle}. There are ten cases. In each case, we will check that the exchange relation from flipping $\alpha$ holds in $\cC(\Sigma_{g, n})$. 

We will be applying the following observation repeatedly. If $\alpha$ and $\alpha'$ are two compatible arcs forming a dangle with a jewel $v$ (as in Figure~\ref{fig:dangle}), then $t_{v}(\alpha) \neq t_{v}(\alpha')$ and $t_{v}(\alpha) + t_{v}(\alpha') = 1$.  But in all other cases, if $\alpha$ and $\alpha'$ are two compatible arcs that have a common endpoint at $v$, and $v$ is not the jewel of a dangle, then $t_{v}(\alpha) = t_{v}(\alpha')$. In particular, a tagged triangulation determines a single tagging $t_{v}$ (independent from $\alpha$) for the vertex $v$, provided $v$ is not the jewel of a dangle in the triangulation.  

{\sf Case 1.} The arc $\alpha$ is the unique common edge of two puzzle pieces of type A. 

The two triangles glued along $\alpha$ form a quadrilateral.  Say the edges are $e_{1}, e_{2}, e_{3}, e_{4}$ in counterclockwise order, and $e_{1}$ and $e_{4}$ are adjacent to $v$. Figure \ref{fig:mutation1} describes the configuration of the arcs, but with the tags suppressed  at the four vertices. Let $\alpha'$ be the flip of $\alpha$.  We need to check that $\rho$ preserves the exchange relation $\alpha \alpha' = e_{1}e_{3} + e_{2}e_{4}$.  

\begin{figure}[!ht]
\begin{minipage}{1.2in}\includegraphics[width=\textwidth]{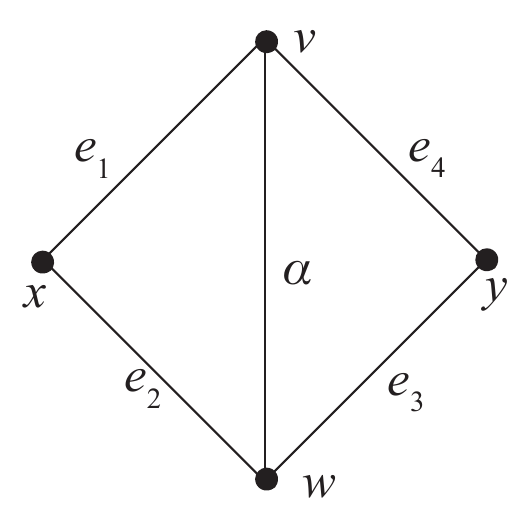}\end{minipage}
\quad $\longrightarrow$ \quad
\begin{minipage}{1.2in}\includegraphics[width=\textwidth]{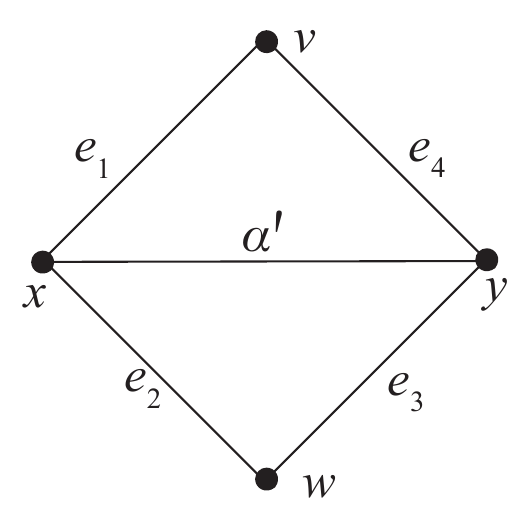}\end{minipage} 
\caption{In Case 1,  two type A puzzle pieces are glued along exactly one edge $\alpha$. The induced exchange relation from flipping $\alpha$ is $\alpha \alpha' = e_{1}e_{3} + e_{2}e_{4}$}
\label{fig:mutation1}
\end{figure}

Although we have not shown the taggings, we know that on the left $t_{v}:= t_{v}(\alpha) = t_{v}(e_1)= t_{v}(e_3)$ and $t_{w}:= t_{w}(\alpha) = t_{w}(e_2) = t_{w}(e_3)$, and on the right  $ t_{x} := t_{x}(\alpha') = t_{x}(e_1)= t_{x}(e_2)$ and $ t_{y}:=t_{y}(\alpha') = t_{y}(e_3) = t_{y}(e_4)$ by our earlier observation about the compatibility in the absence of dangles.

By definition of $\rho$, we have $\rho(\alpha \alpha') = \rho(\alpha) \rho(\alpha')  =v^{t_{v}}w^{t_{w}}x^{t_{x}}y^{t_{y}} \underline{\alpha}\, \underline{\alpha'}  $.  
Similarly, $ \rho(e_{1}e_{3} )= v^{t_{v}}w^{t_{w}}x^{t_{x}}y^{t_{y}} \underline{e_{1}}\, \underline{e_{3}} $ and $\rho(e_{2}e_{4} )= v^{t_{v}}w^{t_{w}}x^{t_{x}}y^{t_{y}} \underline{e_{2}}\, \underline{e_{4}} $.   

In $\cC(\Sigma_{g, n})$, we have $\underline{\alpha}\, \underline{\alpha'} = \underline{e_{1}} \, \underline{e_{3}} + \underline{e_{2}} \, \underline{e_{4}}$ by the skein relation (1) in Definition \ref{def:curvealgebra}. Thus 
 \[ \rho(\alpha\alpha')  = v^{t_{v}}w^{t_{w}}x^{t_{x}}y^{t_{y}} \underline{ \alpha'} \underline{\alpha} 
 = v^{t_{v}}w^{t_{w}}x^{t_{x}}y^{t_{y}}(
 \underline{e_{1}}\, \underline{e_{3}} + \underline{e_{2}} \, \underline{e_{4}} )= \rho(e_{1}e_{3} + e_{2}e_{4}).
\]

{\sf Case 2.} The arc $\alpha$ is one of two common edges of two puzzle pieces of type A. 

In this case the two triangles form a one-punctured bigon, as in the left of Figure \ref{fig:mutation2}. 
Flipping $\alpha$ produces the figure on the right, with the tags suppressed for simplicity. If both  $\alpha$ and $e_{2}$ are plain at $w$, then flipping $\alpha$ produces $\alpha'$ notched at $w$ while $e_{2}$ remains plain at $w$, as depicted in Figure \ref{fig:mutation2}.  But if both $\alpha$ and $e_{2}$ are notched at $w$, then flipping $\alpha$ produces $\alpha'$ plain at $w$ while $e_{2}$ remains notched at $w$. The taggings at $v$ and $x$ are unchanged by the flip.   

\begin{figure}[!ht]
\begin{minipage}{1.2in}\includegraphics[width=\textwidth]{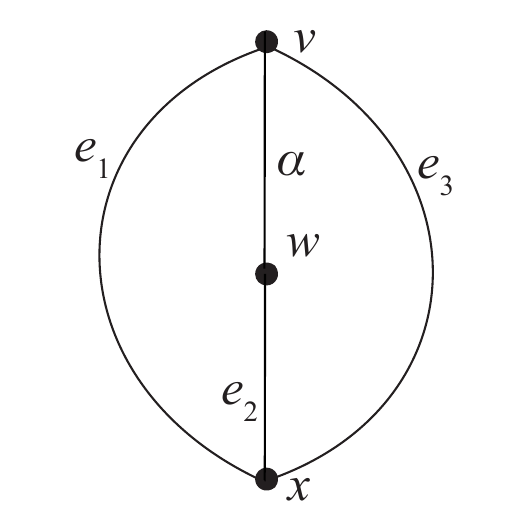}\end{minipage}
\quad $\longrightarrow$ \quad
\begin{minipage}{1.2in}\includegraphics[width=\textwidth]{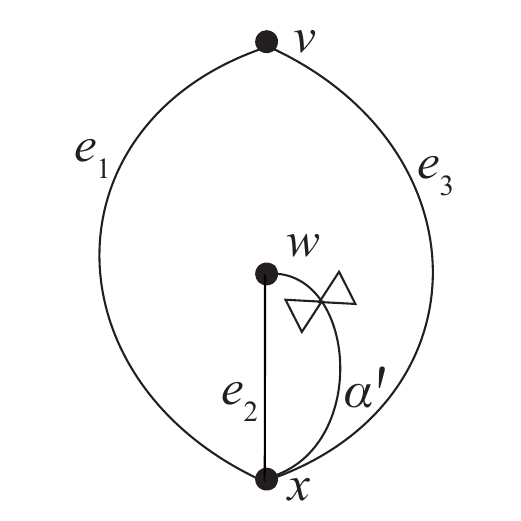}\end{minipage} 
\caption{In Case 2, two type A puzzle pieces are glued along two edges, and the one labeled $\alpha$ is flipped. The cluster mutation is 
$\alpha \alpha' = e_{1} + e_{3}$.}
\label{fig:mutation2}
\end{figure}

Since the tags are all the same at $v$, we denote the tagging of any arc ending at $v$ simply by $t_{v}$, and similarly we use $t_{x}$ for $x$. At $w$, we have $t_{w}(\alpha) = t_{w}(e_{2})$, but $t_{w}(e_{2}) \neq t_{w}(\alpha')$ and $t_{v}(\alpha) + t_{v}(\alpha') = 1$. So $\rho(\alpha \alpha') =  wv^{t_{v}}x^{t_{x}}\underline{\alpha}\underline{\alpha'}$.  Furthermore, note that $\underline{\alpha'} = \underline{e_{2}}$, and by the puncture-skein relation in Definition \ref{def:curvealgebra}, we have $w\underline{\alpha} \underline{e_{2}} = \underline{e_{1}} + \underline{e_{3}}$. Thus, 
\[
	\rho(\alpha \alpha') 
	= wv^{t_{v}}x^{t_{x}}\underline{\alpha}\underline{e_{2}}
	= v^{t_{v}}x^{t_{x}}( \underline{e_{1}} +\underline{e_{3}}) 
	= v^{t_{v}}x^{t_{x}}\underline{e_{1}} + v^{t_{v}}x^{t_{x}} \underline{e_{3}} = \rho(e_{1} + e_{3}).
\]

{\sf Case 3.} The arc $\alpha$ is one of three common edges of two puzzle pieces of type A. 

In this case, $\Sigma_{g, n} = \Sigma_{0, 3}$, which is excluded by assumption. 

{\sf Case 4.} The arc $\alpha$ is the unique common edge of two puzzle pieces of type A and B. 

The result of gluing the two puzzle pieces is shown in Figure \ref{fig:mutation3}. 
\begin{figure}[!ht]
\begin{minipage}{1.2in}\includegraphics[width=\textwidth]{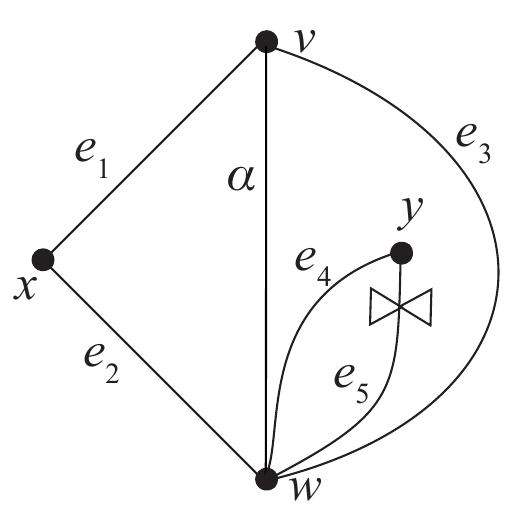}\end{minipage}
\quad $\longrightarrow$ \quad
\begin{minipage}{1.2in}\includegraphics[width=\textwidth]{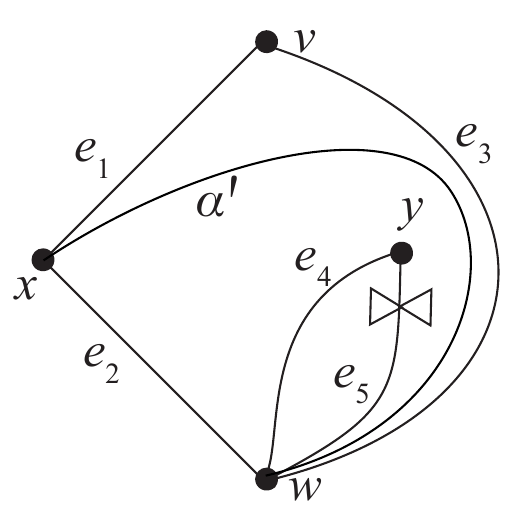}\end{minipage} 
\caption{In Case 4, a type A puzzle piece is glued to a type B puzzle piece along exactly one edge, $\alpha$.  The exchange relation from flipping $\alpha$ is $\alpha \alpha' = e_{1}e_{4}e_{5} + e_{2}e_{3}$ (Lemma \ref{lem:monogon}).}
\label{fig:mutation3}
\end{figure}

Again by compatibility, we denote the tagging of any arc ending at $v$, $w$, and $x$ by $t_{v}$, $t_{w}$, and $t_{x}$, respectively.  Also, exactly one of $e_{4}$ and $e_{5}$ is notched at $y$. 
Thus $\rho(\alpha \alpha') = v^{t_{v}}w^{2t_{w}}x^{t_{x}}\underline{\alpha} \underline{\alpha'}$,
$\rho(e_{1}e_{4}e_{5}) = v^{t_{v}}x^{t_{x}}  w^{2t_{w}} y\underline{e_{1}} \underline{e_{4}}\underline{e_{5}}$, and 
$\rho(e_{2}e_{3}) = v^{t_{v}}w^{2t_{w}}x^{t_{x}}\underline{e_{2}}\underline{e_{3}}$.  

 In $\cC(\Sigma_{g, n})$, application of a skein relation implies $\underline{\alpha} \underline{\alpha'} = \underline{e_{1}}\underline{\gamma_{y}^{w}} + \underline{e_{2}}\underline{e_{3}}$, where $\gamma_{y}^{w}$ is the envelope of the dangle $d_{y}^{w}$ (Definition \ref{def:dangle}). Lemma \ref{lem:monogon} further shows $\underline{\gamma_{y}^{w}} = y\underline{e_{4}}^{2} $, and since  the underlying curves of $e_{4}$ and $e_{5}$ are the same, in fact $\underline{\gamma_{y}^{w}} = y \underline{e_{4}}\,\underline{e_{5}}$. It follows that 
\[
	\rho(\alpha \alpha') 
	= v^{t_{v}}w^{2t_{w}}x^{t_{x}}\underline{\alpha} \underline{\alpha'}
	 = v^{t_{v}}w^{2t_{w}}x^{t_{x}}(\underline{e_{1}} \; y\underline{e_{4}}\underline{e_{5}}+ \underline{e_{2}}\underline{e_{3}}) 
	= \rho(e_{1}e_{4}e_{5} + e_{2}e_{3}).\]

{\sf Case 5.} The arc $\alpha$ is one of two common edges of two puzzle pieces of type A and B. 

Figure \ref{fig:mutation4} shows the two puzzle pieces glued along $\alpha$. 
\begin{figure}[!htpb]
\begin{minipage}{1.2in}\includegraphics[width=\textwidth]{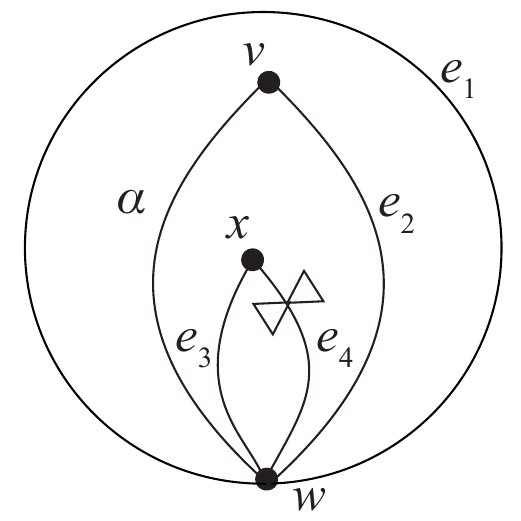}\end{minipage}
\quad $\longrightarrow$ \quad
\begin{minipage}{1.2in}\includegraphics[width=\textwidth]{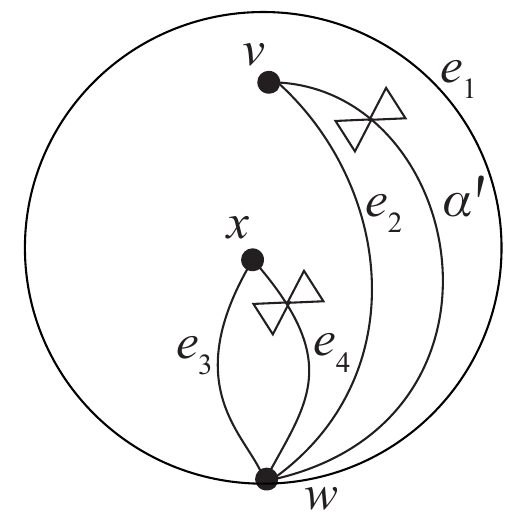}\end{minipage} 
\caption{In Case 5, a type A puzzle is glued to a type B puzzle along two edges, and we flip the one labeled $\alpha$. The exchange relation is $\alpha \alpha' = e_{1} + e_{3}e_{4}$. }
\label{fig:mutation4}
\end{figure}

Note that $\alpha$ and $\alpha'$ have different tags at $v$, and $e_{3}$ and $e_{4}$ have different tags at $x$.  The puncture-skein relation and Lemma \ref{lem:monogon} imply that $v \underline{\alpha}\underline{\alpha'} = \underline{e_{1}} + x\underline{e_{3}}^{2}$.  Since $\underline{e_{3}} = \underline{e_{4}}$, it follows that 
\[
	\rho(\alpha\alpha')  = vw^{2t_{w}}\underline{\alpha}\underline{\alpha'}
	=  w^{2t_{w}}\underline{e_{1}} + x w^{2t_{w}} \underline{e_{3}}\underline{e_{4}} 
	= \rho(e_{1} + e_{3}e_{4}).
\]

{\sf Case 6.} The arc $\alpha$ is the common edge of two puzzle pieces of type A and C. 

Figure \ref{fig:mutation5} shows the two puzzle pieces.  Similarly to the previous cases, 
\[
\begin{split}
	\rho(\alpha \alpha') &= w^{3t_{w}}z^{t_{z}}\underline{\alpha}\underline{\alpha'} = w^{3t_{w}}z^{t_{z}}(\underline{e_{1}}\underline{\gamma_{x}^{w}} + \underline{e_{2}}\underline{\gamma_{y}^{w}}) = w^{3t_{w}}z^{t_{z}}(\underline{e_{1}}x\underline{e_{3}}^{2}+\underline{e_{2}}y\underline{e_{5}}^{2})\\
	&= w^{3t_{w}}z^{t_{z}}(\underline{e_{1}}x\underline{e_{3}}\underline{e_{4}} + \underline{e_{2}}y\underline{e_{5}}\underline{e_{6}}) = \rho(e_{1}e_{3}e_{4}+e_{2}e_{5}e_{6}).
\end{split}
\]
\begin{figure}[!ht]
\begin{minipage}{1.2in}\includegraphics[width=\textwidth]{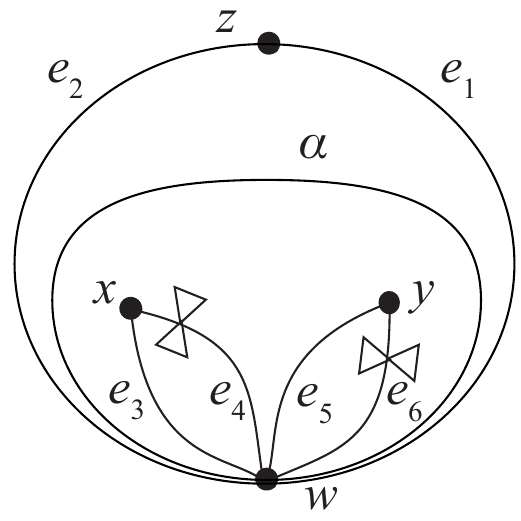}\end{minipage}
\quad $\longrightarrow$ \quad
\begin{minipage}{1.2in}\includegraphics[width=\textwidth]{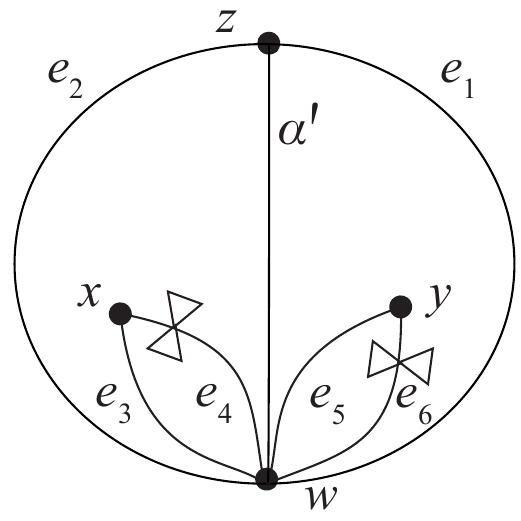}\end{minipage} 
\caption{In Case 6, a type A puzzle piece is glued to a type C puzzle piece along exactly one arc $\alpha$.  The exchange relation is $\alpha\alpha' = e_{1}e_{3}e_{4} + e_{2}e_{5}e_{6}$. }
\label{fig:mutation5}
\end{figure}

{\sf Case 7.} The arc $\alpha$ is the common edge of two puzzle pieces of type B.  

There are two possibilities.  The first one is identical to the right figure in Figure \ref{fig:mutation5}, but where $\alpha'$ plays the role of $\alpha$.  The exchange relation is the same as in Case 6, since the cluster mutation is involutive.  Thus the argument from \textsf{Case 6} applies in this case.  

The second possibility is the one shown in Figure \ref{fig:mutation6}. Then
\[
\begin{split}
	\rho(\alpha\alpha') &= v^{2t_{v}}w^{2t_{w}}\underline{\alpha}\underline{\alpha'} = v^{2t_{v}}w^{2t_{w}}(\underline{e_{1}}\underline{e_{2}} + \underline{\gamma_{x}^{v}}\underline{\gamma_{y}^{w}}) = v^{2t_{v}}w^{2t_{w}}(\underline{e_{1}}\underline{e_{2}} + (x\underline{e_{3}}^2)(y\underline{e_{6}}^2))\\
	&= v^{2t_{v}}w^{2t_{w}}(\underline{e_{1}}\underline{e_{2}} + xy\underline{e_{3}}\underline{e_{4}}\underline{e_{5}}\underline{e_{6}}) = \rho(e_{1}e_{2} + e_{3}e_{4}e_{5}e_{6}).
\end{split}
\]

\begin{figure}[!ht]
\begin{minipage}{1.2in}\includegraphics[width=\textwidth]{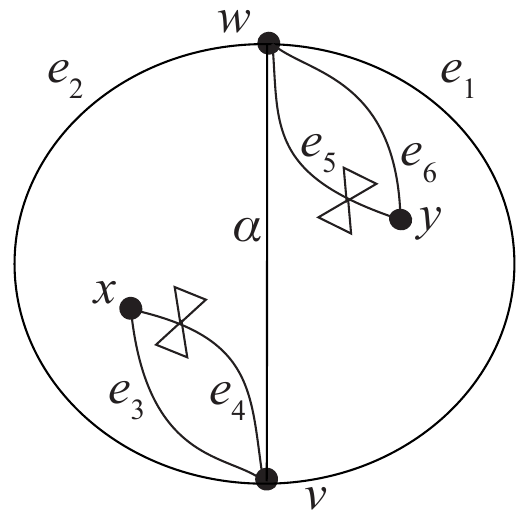}\end{minipage}
\quad $\longrightarrow$ \quad
\begin{minipage}{1.2in}\includegraphics[width=\textwidth]{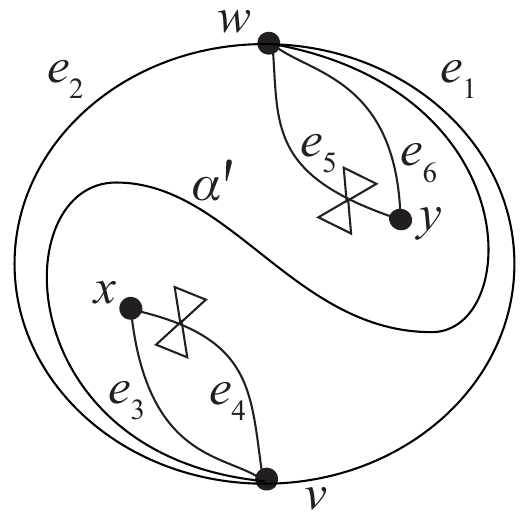}\end{minipage} 
\caption{In Case 7, two puzzle pieces of type B are glued along $\alpha$. In one xe depicted here, the exchange relation is
$\alpha\alpha' = e_{1}e_{2}+e_{3}e_{4}e_{5}e_{6}$.}
\label{fig:mutation6}
\end{figure}

{\sf Case 8.} The arc $\alpha$ is one of two common edges of two puzzle pieces of type B. 

Two puzzle pieces glue together to produce a triangulation for $\Sigma_{0, 4}$.   We distinguish between two subcases, as depicted in Figures~\ref{fig:mutation7} and ~\ref{fig:mutation8}. 

In subcase I shown in Figure \ref{fig:mutation7}, we have
\[
	\rho(\alpha\alpha') = vw^{2t_{w}}\underline{\alpha}\underline{\alpha'} = w^{2t_{w}}(\underline{\gamma_{x}^{w}}+\underline{\gamma_{y}^{w}}) = w^{2t_{w}}(x\underline{e_{1}}\underline{e_{2}} + y\underline{e_{3}}\underline{e_{4}}) = \rho(e_{1}e_{2} + e_{3}e_{4}).
\]
Note that $\underline{\alpha}\underline{\alpha'} = \underline{\gamma_{x}^{w}}+\underline{\gamma_{y}^{w}}$ because it is on $\Sigma_{0, 4}$. In subcase II shown in Figure \ref{fig:mutation8}, we have 
\[
	\rho(\alpha\alpha') = v^{2t_{v}}w^{2t_{w}}\underline{\alpha}\underline{\alpha'} = v^{2t_{v}}w^{2t_{w}}(\underline{\gamma_{x}^{v}}\underline{\gamma_{y}^{w}} + \underline{e_{1}}^{2}) = v^{2t_{v}}w^{2t_{w}}(x\underline{e_{2}}\underline{e_{3}}y\underline{e_{4}}\underline{e_{5}} + \underline{e_{1}}^{2}) = \rho(e_{2}e_{3}e_{4}e_{5} + e_{1}^{2}).
\]

\begin{figure}[!ht]
\begin{minipage}{1.8in}\includegraphics[width=\textwidth]{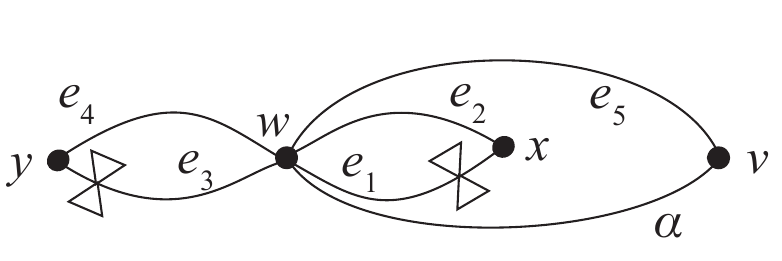}\end{minipage}
\quad $\longrightarrow$ \quad
\begin{minipage}{1.8in}\includegraphics[width=\textwidth]{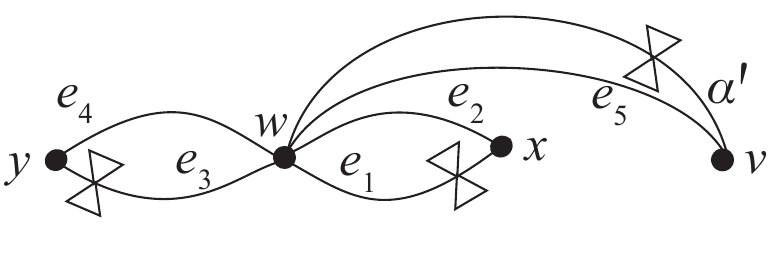}\end{minipage} 
\caption{In subcase I of Case 8, two puzzle pieces are glued to produce a triangulation for $\Sigma_{0,4}$, and 
the exchange relation from flipping $\alpha$ is $\alpha\alpha' = e_{1}e_{2} + e_{3}e_{4}$.  }
\label{fig:mutation7}
\end{figure}

\begin{figure}[!ht]
\begin{minipage}{1.8in}\includegraphics[width=\textwidth]{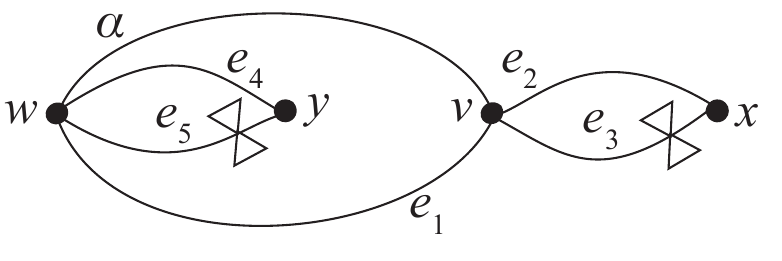}\end{minipage}
\quad $\longrightarrow$ \quad
\begin{minipage}{1.8in}\includegraphics[width=\textwidth]{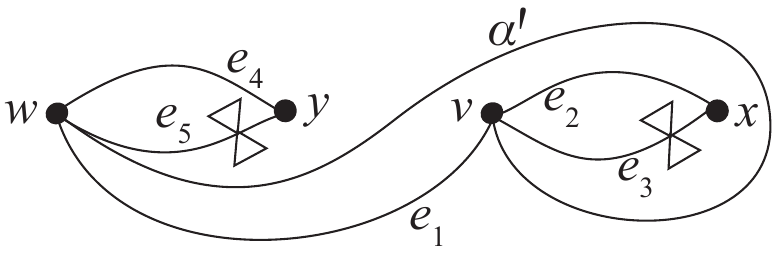}\end{minipage} 
\caption{In subcase II of Case 8, again a triangulation for $\Sigma_{0,4}$ is obtained, and the exchange relation is $\alpha\alpha' = e_{1}^{2}+e_{2}e_{3}e_{4}e_{5}$. The result of the flip is again a union of two puzzle pieces of type B.}
\label{fig:mutation8}
\end{figure}

{\sf Case 9.} The arc $\alpha$ is the common edge of two puzzle pieces of type B and C. 

See Figure \ref{fig:mutation9}.  We have
\[
	\rho(\alpha\alpha') = w^{4t_{w}}\underline{\alpha}\underline{\alpha'} = w^{4t_{w}}(\underline{\gamma_{x}^{w}}\underline{\gamma_{z}^{w}} + \underline{e_{1}}\underline{\gamma_{y}^{w}}) = w^{4t_{w}}(x\underline{e_{2}}\underline{e_{3}}z\underline{e_{6}}\underline{e_{7}}+\underline{e_{1}}y\underline{e_{4}}\underline{e_{5}}) = \rho(e_{2}e_{3}e_{6}e_{7} + e_{1}e_{4}e_{5}).
\]

\begin{figure}[!ht]
\begin{minipage}{1.2in}\includegraphics[width=\textwidth]{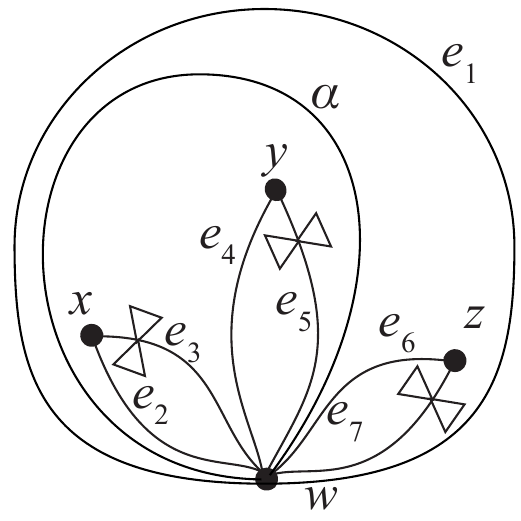}\end{minipage}
\quad $\longrightarrow$ \quad
\begin{minipage}{1.2in}\includegraphics[width=\textwidth]{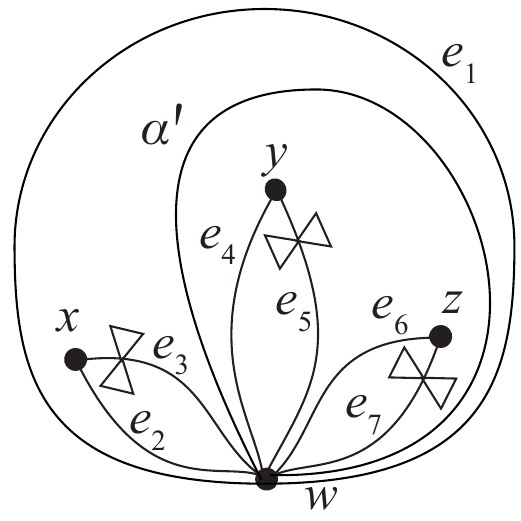}\end{minipage} 
\caption{In Case 9, two puzzle pieces of type B and C are glued along only one edge $\alpha$.  
The exchange relation is $\alpha\alpha' = e_{1}e_{4}e_{5} + e_{2}e_{3}e_{6}e_{7}$. }
\label{fig:mutation9}
\end{figure}

{\sf Case 10.} The arc $\alpha$ is the common edge of two puzzle pieces of type C.

In this situation, the surface must be $\Sigma_{0,5}$.  See Figure \ref{fig:mutation10}. Then
\[
	\rho(\alpha\alpha') = v^{4t_{v}}\underline{\alpha}\underline{\alpha'} = v^{4t_{v}}(\underline{\gamma_{x}^{v}}\underline{\gamma_{z}^{v}} + \underline{\gamma_{y}^{v}}\underline{\gamma_{w}^{v}}) = v^{4t_{v}}(x\underline{e_{1}}\underline{e_{2}}z\underline{e_{5}}\underline{e_{6}}+y\underline{e_{3}}\underline{e_{4}}w\underline{e_{7}}\underline{e_{8}}) = \rho(e_{1}e_{2}e_{5}e_{6} + e_{3}e_{4}e_{7}e_{8}).
\]
\begin{figure}[!ht]
\begin{minipage}{1.2in}\includegraphics[width=\textwidth]{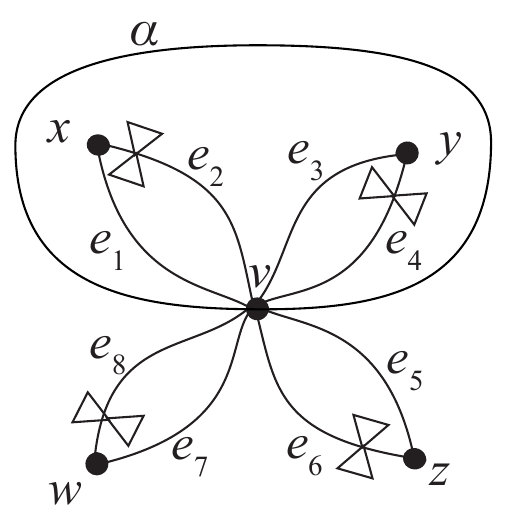}\end{minipage}
\quad $\longrightarrow$ \quad
\begin{minipage}{1.2in}\includegraphics[width=\textwidth]{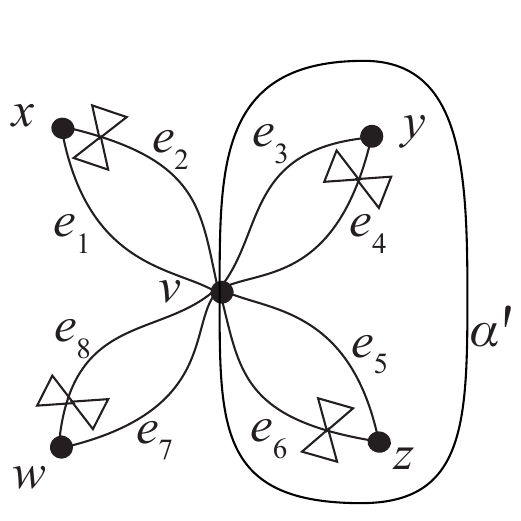}\end{minipage} 
\caption{In Case 10, two puzzle pieces of type C are glued along two edges, and the one labeled $\alpha$ is flipped. 
The exchange relation is $\alpha\alpha' = e_{1}e_{2}e_{5}e_{6} + e_{3}e_{4}e_{7}e_{8}$. 
}
\label{fig:mutation10}
\end{figure}

{\bf Step 2.} Suppose that $\alpha$ is on a dangle. 

Any dangle must be contained inside one of the tagged puzzle pieces in Figure \ref{fig:taggedpuzzle}. Suppose first that $\alpha$ is notched at the jewel. The mutation of $\alpha$ in a puzzle of type B is the inverse of the flip described in \textsf{Case 2} and Figure \ref{fig:mutation2} (and $\alpha'$ in Figure \ref{fig:mutation2} plays the role of $\alpha$). Since the mutation is an involution, the compatibility follows from \textsf{Case 2}. In the case of a puzzle of type C, the mutation is the inverse of the flip in \textsf{Case 5} and Figure \ref{fig:mutation4}. In the case of type D, it is the inverse of the flip in subcase I of \textsf{Case 8} and Figure \ref{fig:mutation7}. This takes care of all situations where $\alpha$ is on a dangle. If $\alpha$ is tagged plainly at the jewel, then the only difference is that, in the flipped diagram, one needs to change the tagging at the vertex which was the jewel. The rest of the computation is identical. 
\end{proof}

\begin{remark}\label{rmk:rhoRcoeff}
By tensoring a commutative ring $R$, we obtain 
\[
	\rho_{R} : \cA(\Sigma_{g, n})_{R} \to \cC(\Sigma_{g, n})_{R}.
\]
\end{remark}

We complete the proof of Compatibility Lemma by showing that $\rho$ is injective. Indeed, we will show that for any integral domain $R$, $\rho_{R}$ in Remark \ref{rmk:rhoRcoeff} is injective.

Roger-Yang's homomorphism $\Phi : \cC(\Sigma_{g, n}) \to C^{\infty}(\cT^{d}(\Sigma_{g, n}))$ will factor in our proofs coming up.  We here present a slightly different version that we find easier to apply. See \cite[Section 3]{MoonWong21} for details.  

\begin{lemma} \label{lemma:phihat}
Let $R$ be an integral domain. Suppose $\cT$ is an ideal triangulation of $\Sigma_{g, n}$, and let $E= \{e_{i} \}_{i= 1}^{m}$ denote its set of edges. Then there is a well-defined homomorphism $\hat \Phi_{R}: \cC(\Sigma_{g, n})_{R} \to Q(R)(e_i)$, where $Q(R)$ is the field of fraction of $R$.
\end{lemma}

\begin{proof}
We first consider $R = \ZZ$ case. The map $\Phi$ sends each arc $e_{i}$ in $\cT$ to the function $\lambda_{i}$ on the decorated Teichm\"uller space $\cT^{d}(\Sigma_{g, n})$ that gives the lambda-length of $e_{i}$. It follows by \cite[Lemma 3.3]{MoonWong21} that $\Phi$ factors through $\Phi : \cC(\Sigma_{g, n}) \to \ZZ[\lambda_{i}^{\pm 1}]$. By tensoring a general integral domain $R$, we obtain a similar map $\Phi_{R} : \cC(\Sigma_{g, n})_{R} \to R[\lambda_{i}^{\pm 1}]$.

The decorated Teichm\"uller space $\cT^{d}(\Sigma_{g, n})$ is homeomorphic to $\RR^{m}_{> 0}$ and the homeomorphism maps each decorated hyperbolic metric $(m, r)$ to the lambda-lengths $\{\lambda_{i}\}$ of $\{e_{i}\}$ \cite[Theorem 3.1]{Penner87}. Thus, $\cT^{d}(\Sigma_{g, n})$ is a Zariski-dense semialgebraic set in an $n$-dimensional complex torus $\spec \CC[\lambda_{i}^{\pm}] \cong (\CC^{*})^{m}$. Therefore, $\{\lambda_{i}\}$ is a set of algebraically independent elements. Hence there is a well-defined, canonical isomorphism $\tau : \ZZ[\lambda_{i}^{\pm}] \cong \ZZ[e_{i}^{\pm}]$ that maps $\lambda_{i}$ to $e_{i}$. By tensoring $R$, we obtain $\tau_{R} : R[\lambda_{i}^{\pm}] \cong R[e_{i}^{\pm}]$. Then composition of $\tau_{R}$ and $\Phi_{R}$ followed by the canonical inclusion $R[e_{i}^{\pm}] \subset Q(R)(e_{i})$ yields $\hat \Phi_{R}$.  
\end{proof}  

\begin{proposition} \label{prop:inj}
Let $R$ be an integral domain. The algebra homomorphism $\rho_{R} : \cA(\Sigma_{g, n})_{R} \to \cC(\Sigma_{g, n})_{R}$ is injective. 
\end{proposition}

\begin{proof}
We fix an ordinary triangulation $\cT$ on $\Sigma_{g, n}$. Let $E = \{e_{i}\}$ be the set of edges in $\cT$. There is a commutative diagram
\[
	\xymatrix{\cA(\Sigma_{g, n})_{R} \ar[rr]^{\rho_{R}} \ar[rd]_{\iota} && \cC(\Sigma_{g, n})_{R} \ar[ld]^{\hat\Phi_{R}}\\ &Q(R)(e_{i}).}
\]
Here $\iota$ is the natural inclusion of the cluster algebra $\cA(\Sigma_{g, n})_{R}$ into its field of fraction, and $\rho_{R}$ is the homomorphism in Remark \ref{rmk:rhoRcoeff}.  The map $\hat \Phi_{R}$ is the Roger-Yang homomorphism from Lemma \ref{lemma:phihat}. For each $e_{i}$,  we have $\iota (e_i) = \hat \Phi_{R} \circ \rho (e_i)$. It follows that $\iota = \hat \Phi_{R} \circ \rho_{R}$, since all of the elements of $\cA(\Sigma_{g, n})_{R}$ can be written as a Laurent polynomial with respect to the cluster variables $e_i$ in a fixed cluster. Since $\iota$ is injective, $\rho_{R}$ must also be injective. 
\end{proof}


\section{Integrality of $\cC(\Sigma_{g, n})$ and its implications }\label{sec:impcurve}

This section is mainly devoted to a proof of integrality of $\cC(\Sigma_{g, n})$ using the injective homomorphism $\rho:\cA(\Sigma_{g, n}) \to \cC(\Sigma_{g, n})$ and techniques from algebraic geometry, in particular dimension theory. For the definition and basic properties of the dimension of algebraic varieties, see \cite[Section 8]{Eisenbud95}. We will need the following lemma from commutative algebra.

Let $k$ be a field and let $R$ be a $k$-algebra, which is an integral domain. The \emph{(Krull) dimension} $\dim R$ of $R$ is the maximal length $\ell$ of the strictly increasing chain of prime ideals $0 = P_{0} \subsetneq P_{1} \subsetneq P_{2} \subsetneq \cdots \subsetneq P_{\ell}$ of $R$. For the associated affine scheme $\spec R$, its dimension is defined as $\dim \spec R = \dim R$. 

\begin{lemma}\label{lem:transdegreeanddim}
Let $k$ be a field and let $R$ be a $k$-algebra, which is an integral domain. Let $Q(R)$ be its field of fractions. Suppose that the transcendental degree $\mathrm{trdeg}_{k}Q(R)$ of $Q(R)$ is $m$. Then $\dim R \le m$. 
\end{lemma}

\begin{proof}
When $R$ is a finitely generated algebra, the statement is well known \cite[Theorem A, p.221]{Eisenbud95}. We assume that $R$ is not finitely generated. 

Take a chain of prime ideals $0 = P_{0} \subsetneq P_{1} \subsetneq P_{2} \subsetneq \cdots \subsetneq P_{\ell}$ of $R$. For each $1 \le i \le \ell$, pick $x_{i} \in P_{i} \setminus P_{i-1}$. Let $R'$ be the subalgebra of $R$ generated by $\{x_{i}\}$, and $Q(R') \subset Q(R)$ be its field of fractions. Since $R'$ is a finitely generated algebra, $\dim R' \le \mathrm{trdeg}_{k}Q(R') \le \mathrm{trdeg}_{k}Q(R) = m$. 

On the other hand, if we set $P_{i}' = P_{i} \cap R'$, the sequence $P_{0}' \subset P_{1}' \subset P_{2}' \subset \cdots \subset P_{k}'$ is an increasing sequence of prime ideals, and it is strictly increasing as $x_{i} \in P_{i}' \setminus P_{i-1}'$. Therefore, $\dim R' \ge \ell$, so we have $\ell \le m$. This is valid for arbitrary increasing chains of prime ideals, we obtain the desired result. 
\end{proof}

We are now ready for the proof of integrality. 

\begin{theorem} \label{thm:intdom}
Suppose that $\chi(\Sigma_{g, n}) = 2 - 2g - n < 0$ and $n > 0$. Then $\cC(\Sigma_{g,n})$ is an integral domain.  
\end{theorem}

\begin{proof}
To start, assume that $\Sigma_{g, n}$ is not a $3$-puncture sphere, so that $\cA(\Sigma_{g, n})$ is defined. As before, we fix an ordinary ideal triangulation $\cT$ on $\Sigma_{g, n}$, and let $E = \{e_{i}\}$ denote the edges of the triangulation. 

By \cite[Lemma 3.2]{MoonWong21}, every element in $\cC(\Sigma_{g, n})$ can be written as a rational function (indeed a Laurent polynomial) with respect to the edge classes $\{e_{i}\}$ in $\cT$. In particular, for any $x \in \cC(\Sigma_{g, n}) \setminus \rho(\cA(\Sigma_{g, n}))$, there is a rational function $f(e_{i})/g(e_{i})$ with respect to $\{e_{i}\}$, such that $x = f(e_{i})/g(e_{i})$. Indeed, the numerator $f$ is not a zero polynomial, because it is given by the trace of a product of matrices whose coefficients are edge classes \cite[Theorem 3.22]{RogerYang14}. Then we can construct a ring extension $\cA(\Sigma_{g, n})' := \cA(\Sigma_{g, n})[t]/(t - f/g)$ and an extended homomorphism $\rho' : \cA(\Sigma_{g, n})' \to \cC(\Sigma_{g, n})$, which maps $t \mapsto x$. Since $t  = f/g \in \QQ(e_{i})$, $\cA(\Sigma_{g, n})'$ is also a subring of $\QQ(e_{i})$. We may repeat this procedure and extend the algebra $\cA(\Sigma_{g, n})'$ further, until the extended map is surjective. Since $\cC(\Sigma_{g, n})$ is a finitely generated algebra (Theorem \ref{thm:finitegeneration}), this procedure is terminated in finitely many steps. Therefore, we obtain a ring extension $\widetilde{\cA}(\Sigma_{g, n})$ of $\cA(\Sigma_{g, n})$ in $\QQ(e_{i})$ and a surjective homomorphism $\widetilde{\rho} : \widetilde{\cA}(\Sigma_{g, n}) \to \cC(\Sigma_{g, n})$. 

Combined with the Roger-Yang homomorphism from Lemma \ref{lemma:phihat} (with $\ZZ$-coefficient), we have the commutative diagram 
\[
	\xymatrix{\cA(\Sigma_{g, n}) \ar[d] \ar[r]^{\rho} & \cC(\Sigma_{g, n}) \ar[d]^{\hat \Phi}\\ \widetilde{\cA}(\Sigma_{g, n}) \ar[ru]^{\widetilde{\rho}}\ar[r] & \QQ(e_{i}).}
\]

Note that $\widetilde{\cA}(\Sigma_{g, n})$ is an integral domain, as it is a subring of $\QQ(e_{i})$. So is $\widetilde{\cA}(\Sigma_{g, n})_{\CC} := \widetilde{\cA}(\Sigma_{g, n}) \otimes_{\ZZ}\CC \subset \QQ(e_{i}) \otimes_{\ZZ}\CC = \CC(e_{i})$.  So the associated affine scheme $\widetilde{\cA}(\Sigma_{g, n})_{\CC}$ is integral (irreducible and reduced). If we denote the number of edges in $\cT$ by $m$, then the transcendental degree is $\mathrm{trdeg}_{\CC}(\CC(e_{i})) = m$. Since the field of fractions of $\widetilde{\cA}(\Sigma_{g, n})_{\CC}$ is also $\CC(e_{i})$, by Lemma \ref{lem:transdegreeanddim}, $\dim \spec \widetilde{\cA}(\Sigma_{g, n})_{\CC} \le m$. 

Since $\widetilde{\rho} : \widetilde{\cA}(\Sigma_{g, n}) \to \cC(\Sigma_{g, n})$ is a surjective homomorphism, so is $\widetilde{\rho}_{\CC} : \widetilde{\cA}(\Sigma_{g, n})_{\CC} := \widetilde{\cA}(\Sigma_{g, n}) \otimes_{\ZZ}\CC \to \cC(\Sigma_{g, n})_{\CC}$. If we denote $\ker \widetilde{\rho}_{\CC} = I$, then $\widetilde{\cA}(\Sigma_{g, n})_{\CC}/I \cong \cC(\Sigma_{g, n})_{\CC}$. Then $\spec \cC(\Sigma_{g, n})_{\CC}$ is a closed subscheme of $\spec \widetilde{\cA}(\Sigma_{g, n})_{\CC}$, defined by the ideal $I$. Thus
\[
	\dim \spec \cC(\Sigma_{g, n})_{\CC} \le \dim \spec \widetilde{\cA}(\Sigma_{g, n})_{\CC} \le m
\]
and if $I$ is nontrivial, then $\dim \spec \cC(\Sigma_{g, n})_{\CC} <  \dim \spec \widetilde{\cA}(\Sigma_{g, n})_{\CC}$.

Recall from the proof of Lemma \ref{lemma:phihat} that $\hat \Phi_{\CC} : \cC(\Sigma_{g, n})_{\CC} \to \CC(e_{i})$ is a composition of the map $\cC(\Sigma_{g, n})_{\CC} \to \CC[\lambda_{i}^{\pm}] \to \CC[e_{i}^{\pm}] \subset \CC(e_{i})$. Every element in $\cC(\Sigma_{g, n})_{\CC}$ can be written as a Laurent polynomial with respect to $\{e_{i}\}$ \cite[Lemma 3.2]{MoonWong21}, so if we denote $S$ by the multiplicative set of monomials with respect to $\{e_{i}\}$, then there is a localized morphism $S^{-1}\cC(\Sigma_{g, n})_{\CC} \to \CC[\lambda_{i}^{\pm}]$, which turns out to be an isomorphism \cite[Lemma 3.4]{MoonWong21}. The localization of a ring corresponds to taking an open subset of the associated scheme. Thus $\spec \cC(\Sigma_{g, n})_{\CC}$ has a (Zariski) open subset $\spec S^{-1}\cC(\Sigma_{g, n})_{\CC} \cong \spec \CC[\lambda_{i}^{\pm 1}] \cong \spec \CC[e_{i}^{\pm}] \cong (\CC^{*})^{m}$. In particular, $\spec \cC(\Sigma_{g, n})_{\CC}$ has an irreducible component, which has an open dense subset isomorphic to the algebraic torus of dimension $m$. Therefore, $\dim \spec \cC(\Sigma_{g, n})_{\CC} \ge m$. 

The only possibility is $\dim \spec \cC(\Sigma_{g, n})_{\CC} = m$ and $\ker \widetilde{\rho} = I$ is the trivial ideal. Therefore $\cC(\Sigma_{g, n})_{\CC} \cong \widetilde{\cA}(\Sigma_{g, n})_{\CC}$ and hence is an integral domain. Since $\cC(\Sigma_{g, n})$ has no torsion (Remark \ref{rem:torsionfree}), $\cC(\Sigma_{g, n}) \subset \cC(\Sigma_{g, n})_{\CC}$ and it is also an integral domain. 

Now the only remaining case is $\Sigma_{0,3}$ where $\cA(\Sigma_{0, 3})$ is undefined. But we may formally set $\cA(\Sigma_{0, 3}) = \ZZ[e_{i}]_{1 \le i \le 3}$ and define $\rho : \cA(\Sigma_{0, 3}) \to \cC(\Sigma_{0,3})$ as $\rho(e_{i}) = \underline{\beta_{ii+1}}$ (see Theorem \ref{thm:presentationC0n} for the notation). Then we can follow the same line of the proof to get the same conclusion. 
\end{proof}

\begin{remark}
If $\chi(\Sigma_{g, n}) \ge 0$ (so $g = 0$ and $n = 1, 2$), $\cC(\Sigma_{g, n})$ is no longer an integral domain \cite[Remark 6.3]{ACDHM21}. 
\end{remark}

The following statement immediately follows from Theorem \ref{thm:intdom} and \cite[Theorem C]{MoonWong21}. 

\begin{theorem}\label{thm:domainSq}
Suppose that $\chi(\Sigma_{g, n}) < 0$ and $n > 0$. Then $\cS^{q}(\Sigma_{g, n})$ is a non-commutative domain. 
\end{theorem}

\begin{remark}
Thang Le kindly informed us that with his collaborators, he also proved Theorem \ref{thm:domainSq} with a completely independent method \cite{BKL23}. Their proof covers even the case that $q$ is not a formal variable.
\end{remark}

\begin{proof}[Proofs of Theorem \ref{thm:defquantTeich} and Theorem \ref{thm:defquantcluster}]
Theorem A of \cite{MoonWong21} states that if $\cC(\Sigma_{g, n})$ is an integral domain, $\Phi$ must be injective. Thus, we obtain Theorem \ref{thm:defquantTeich}. In the last part of the proof of \ref{thm:intdom}, we showed that $\widetilde{\cA}(\Sigma_{g, n})_{\CC} \cong \cC(\Sigma_{g, n})_{\CC}$, so they have the same field of fractions. Since $\widetilde{\cA}(\Sigma_{g, n})_{\CC}$ is an algebraic extension of $\cA(\Sigma_{g, n})_{\CC}$ in its field of fractions, they have the same field of fractions, too. Thus, we can conclude that $\cS^{q}(\Sigma_{g, n})$ can be understood as a deformation quantization of $\cA(\Sigma_{g, n})$. 
\end{proof}


\section{Implications for $\cA(\Sigma_{g, n})$}\label{sec:clusteralgebraimplication}

The compatibility of the curve algebra $\cC(\Sigma_{g, n})$ and cluster algebra $\cA(\Sigma_{g, n})$ provides us new insight to some questions on the structure of cluster algebras. In this section, we investigate two questions regarding the finite generation of $\cA(\Sigma_{g, n})$ (Theorem \ref{thm:infinitegenerationA}) and the comparison of $\cA(\Sigma_{g, n})$ with $\cU(\Sigma_{g, n})$ (Theorem \ref{thm:torusupper}). We still assume that $\chi(\Sigma_{g, n}) < 0$. 

\subsection{Non-finite generation for $g \ge 1$}

It was observed in \cite[Proposition 1.3]{Ladkani13}, following \cite[Proposition 11.3]{Muller13}, that $\cA(\Sigma_{g, 1})$ is not finitely generated for all $g \ge 1$. It is plausible to believe that $\cA(\Sigma_{g, n})$ is more complicated than $\cA(\Sigma_{g,1})$. Thus one may guess that $\cA(\Sigma_{g, n})$ is not finitely generated for all $n$. However, the lack of a functorial morphism $\cA(\Sigma_{g, n}) \to \cA(\Sigma_{g, 1})$ makes it difficult to prove the non-finite generation of $\cA(\Sigma_{g, n})$ in general. We suggest a new approach to resolve this issue, using invariant theory and `mod 2 reduction.'

The first key technical ingredient is Nagata's theorem \cite[Theorem 3.3]{Dolgachev03} and its extension to arbitrary base ring by Seshadri \cite{Seshadri77}. For a finitely generated $k$-algebra $A$, it is not true that its subalgebra $B \subset A$ is finitely generated. However, if $A$ is equipped with a reductive group $G$-action, Nagata's theorem tells us that the invariant subalgebra $A^{G}$ is finitely generated. For our purpose, the following consequence of the Seshadri-Nagata's theorem is handy. 

\begin{lemma}\label{lem:Nagata}
Let $k$ be a field. Let $A$ be a finitely generated $\ZZ^{r}$-graded $k$-algebra, so $A \cong \bigoplus_{\ba \in \ZZ^{r}}A_{\ba}$ such that $A_{\ba}A_{\bfb} \subset A_{\ba+\bfb}$. Then $A_{\mathbf{0}}$ is finitely generated. 
\end{lemma}

\begin{proof}
Recall that an affine group scheme $\mathbb{G}_{m}^{r} := \spec k[x_{i}^{\pm}]_{1 \le i \le r}$-action on $\spec A$ is given by a $k$-linear map
\[
	A \to A \otimes_{k}k[x_{i}^{\pm}],
\]
which makes $A$ as a comodule under the coalgebra $k[x_{i}^{\pm}]$. We may set 
\[
	A_{\ba} := \{r \in A\;|\; r \mapsto r \otimes \prod x_{i}^{a_{i}}\}.
\]
Then it is straightforward to check that the above coalgebra strucure is equivalent to a $\ZZ^{r}$-grading structure on $A$. Now $A_{\mathbf{0}} \cong A^{\mathbb{G}_{m}^{r}}$, which is finitely generated by \cite[Remark 4, p.242]{Seshadri77}. 
\end{proof}

\begin{remark}
The group action in the proof of Lemma \ref{lem:Nagata} should be understood as an affine group \emph{scheme} action, not a set-theoretic one. We will consider the $k = \ZZ_{2}$ case.  But then the set of $\ZZ_{2}$-valued points of $\mathbb{G}_{m}^{r} = \spec \ZZ_{2}[x_{i}^{\pm}]$ has only one point $(1, 1, \cdots, 1)$. Thus, set-theoretically, it is a trivial group. 
\end{remark}

\begin{remark}
Primarily, we will use the contrapositive of Lemma \ref{lem:Nagata}. If $A_{\mathbf{0}}$ is not finitely generated, then  $A$ is not finitely generated. 
\end{remark}

\begin{proposition}\label{prop:grading}
Let $R$ be an integral domain. Then $\cA(\Sigma_{g, n})_{R}$ and $\cC(\Sigma_{g, n})_{R}$ have $\ZZ^{n}$-graded ring structure. 
\end{proposition}

\begin{proof}
We may impose $\cC(\Sigma_{g, n})_{R}$ as a $\ZZ^{n}$-graded algebra structure in the  following way. Let $V = \{v_{i}\}$ be the vertex set. For an arc $\alpha$ connecting $v_{i}$ and another vertex $v_{j}$ (we allow $i = j$), the grade of $\alpha$ is defined as $\be_{i}+\be_{j}$, where $\{\be_{i}\}$ is the standard basis of $\ZZ^{n}$. For any loop, its grade is $\mathbf{0}$. Finally, the grade of the vertex class $v_{i}$ is $-2\be_{i}$ (hence the grade of $v_{i}^{-1}$ is $2\be_{i}$). It is straightfoward to check that all skein relations in Definition \ref{def:curvealgebra} are homogeneous. Thus, it is well-defined. 

Since $\cA(\Sigma_{g, n})_{R} \subset \cC(\Sigma_{g, n})_{R}$ is generated by homogeneous elements, $\cA(\Sigma_{g, n})_{R}$ is also a $\ZZ^{n}$-graded algebra. 
\end{proof}

\begin{remark}\label{rem:Zgrading}
For each vertex $v_{i}$, we may impose a $\ZZ$-graded algebra structure on $\cC(\Sigma_{g, n})_{R}$ (and on $\cA(\Sigma_{g, n})_{R}$), by composing  the grade map with the $i$-th projection $p_{i}: \ZZ^{n} \to \ZZ$. 
\end{remark}

The following proposition is proved by Ladkani in \cite[Proposition 1.3]{Ladkani13}, over $\ZZ$ coefficients. The same proof works for arbitrary base ring, but we provide a sketch for the sake of completeness. 

\begin{proposition}[Ladkani]\label{prop:onepuncturenonfinitegeneration}
For any integral domain $R$ and $g \ge 1$, $\cA(\Sigma_{g, 1})_{R}$ is not finitely generated. 
\end{proposition}

\begin{proof}
By definition when $n = 1$, $\cA(\Sigma_{g, 1})_{R}$ is generated by ordinary arcs only, and  all exchange relations are homogeneous of degree two with respect to the $\ZZ$-grading in Proposition \ref{prop:grading}. Therefore, a cluster variable cannot be expressed as a polynomial with respect to the other cluster variables. On the other hand, there are infinitely many non-isotopic arc classes on $\Sigma_{g, 1}$, so there are infinitely many cluster variables. Thus, $\cA(\Sigma_{g, 1})_{R}$ cannot be finitely generated. 
\end{proof}

Over $\ZZ_{2}$-coefficient, we may reduce the proof of the non-finite generation to $n = 2$ case. 

\begin{proposition}\label{prop:reduction}
Let $n \ge 2$. If $\cA(\Sigma_{g, n})_{\ZZ_{2}}$ is not finitely generated, then $\cA(\Sigma_{g, n+1})_{\ZZ_{2}}$ is not finitely generated. 
\end{proposition}

\begin{proof}
We think of $\cA(\Sigma_{g, n})$ as a subalgebra of $\cC(\Sigma_{g, n})$. Thus, instead of arcs and tagged arcs, we will describe all elements as a combination of arcs and vertices. 

We construct a morphism between curve algebras induced from $\iota : \Sigma_{g, n+1} \to \Sigma_{g, n}$ which forgets a vertex $v$. With respect to $v$, note that $\cC(\Sigma_{g, n+1})_{R}$ and $\cA(\Sigma_{g, n+1})_{R}$ have a $\ZZ$-graded structure (Remark \ref{rem:Zgrading}). Let $\cC(\Sigma_{g, n+1})_{R, 0}$ be the grade 0 subalgebra of $\cC(\Sigma_{g, n+1})_{R}$. 

We claim that when $R = \ZZ_{2}$, there is a well-defined surjective homomorphism $\psi : \cC(\Sigma_{g, n+1})_{\ZZ_{2}, 0} \to \cC(\Sigma_{g, n})_{\ZZ_{2}}$. Indeed, $\cC(\Sigma_{g, n+1})_{\ZZ_{2}, 0}$ is generated by the following elements:
\begin{enumerate}
\item vertex classes $v_{1}^{\pm}, v_{2}^{\pm}, \cdots, v_{n}^{\pm}$ (except $v$); 
\item loop classes;
\item tagged arcs disjoint from $v$;
\item $v\underline{\alpha_{1}}\underline{\alpha_{2}}$ where each $\underline{\alpha_{1}}, \underline{\alpha_{2}}$ are arcs connecting $v$ with other vertices; 
\item $v\underline{\beta}$ where $\underline{\beta}$ is an arc connecting $v$ and itself. 
\end{enumerate}

For each case, by applying a puncture-skein relation, we can find a representative which is disjoint from $v$. For (1), (2), and (3), this is clear. For (4), by the puncture-skein relation, we can resolve the crossing of $v\underline{\alpha_{1}}\underline{\alpha_{2}}$ to get the sum of two arcs disjoint from $v$, which we call  $\underline{\gamma_{1}} $ and $\underline{\gamma_{2}}$. Now if we forget $v$, then as isotopy classes on $\Sigma_{g, n}$, we have $\underline{\gamma_{1}} = \underline{\gamma_{2}}$.   Thus  $v\underline{\alpha_{1}}\underline{\alpha_{2}} = \underline{\gamma_{1}} + \underline{\gamma_{2}} = 2\underline{\gamma_{1}} = 0 \in \cC(\Sigma_{g, n})_{\ZZ_{2}}$. The case of (5) is similar. Since we only used the puncture-skein relation, the map $\psi$ is well-defined. The surjectivity is immediate. 

By composition, we obtain a map 
\[
	\cA(\Sigma_{g, n+1})_{\ZZ_{2}, 0} \to \cC(\Sigma_{g, n+1})_{\ZZ_{2}, 0} \stackrel{\psi}{\to} \cC(\Sigma_{g, n})_{\ZZ_{2}}.
\]
The cluster algebra $\cA(\Sigma_{g, n+1})_{\ZZ_{2}, 0}$ is generated by multiples of tagges arcs, and the image of them by the map $\psi$ is still a multiple of tagged arcs on $\Sigma_{g, n}$. The only exception is a multiple of $v\underline{\beta}$, where $\underline{\beta}$ is an arc whose both ends are $v$. (Note that two ends of $\beta \in \cA(\Sigma_{g, n+1})$, whose underlying curve is $\underline{\beta}$, must be tagged in the same way, so $\beta = \underline{\beta}$ or $\beta = v^{2}\underline{\beta}$.) In this case, after applying the puncture-skein relation at the endpoint of $\underline{\beta}$, $v\underline{\beta}$ becomes a multiple of the sum of two loops $\underline{\ell_{1}}$ and $\underline{\ell_{2}}$.  Once we forget the vertex, then in  $\cC(\Sigma_{g, n})_{\ZZ_{2}}$ we have $\underline{\ell_{1}} + \underline{\ell_{2}} = 2\underline{\ell_{1}} = 0$. In summary, the image of $\cA(\Sigma_{g, n+1})_{\ZZ_{2}, 0}$ by $\psi$ is still tagged arcs on $\Sigma_{g, n}$. Therefore, \emph{if $n \ge 2$}, the image is in $\cA(\Sigma_{g, n})_{\ZZ_{2}}$, and we have a morphism $\psi : \cA(\Sigma_{g, n+1})_{\ZZ_{2}, 0} \to \cA(\Sigma_{g, n})_{\ZZ_{2}}$. 

It is straightforward to check that $\psi : \cA(\Sigma_{g, n+1})_{\ZZ_{2}, 0} \to \cA(\Sigma_{g, n})_{\ZZ_{2}}$ is surjective. Therefore, if $\cA(\Sigma_{g, n})_{\ZZ_{2}}$ is not finitely generated, then $\cA(\Sigma_{g, n+1})_{\ZZ_{2}, 0}$ is not finitely generated. By Lemma \ref{lem:Nagata}, $\cA(\Sigma_{g, n+1})_{\ZZ_{2}}$ is not finitely generated, too. 
\end{proof}

\begin{remark}\label{rem:n=1}
On the other hand, when $n = 1$, $\cA(\Sigma_{g, 1})_{\ZZ_{2}}$ is generated by \emph{ordinary arcs} only. Thus $\psi : \cA(\Sigma_{g, 2})_{\ZZ_{2}, 0} \to \cC(\Sigma_{g, 1})_{\ZZ_{2}}$ does \emph{not} factor through $\cA(\Sigma_{g, 1})_{\ZZ_{2}}$ in general.
\end{remark}

\begin{remark}
The reduction map $\Psi : \cA(\Sigma_{g, n+1})_{R, 0} \to \cA(\Sigma_{g, n})_{R}$ does not behave well for a general base ring $R$. For example, both the punctured loop $\ell$ around $v$ and a trivial loop $\ell'$ near $v$ both map to the same trivial loop under the map that forgets $v$.   Thus $\ell - \ell' = 4$ but after the forgetful map $\Psi(4) = \Psi(\ell - \ell') = 0$. In particular, if the base ring $R$ is a field of characteristic $\ne 2$, $\Psi$ is a zero map. 
\end{remark}

\begin{proof}[Proof of Theorem \ref{thm:infinitegenerationA} for $g \ge 1$] 
\textsf{Step 1.} First of all, observe that to show the non-finite generation of $\cA(\Sigma_{g, n})$, it is sufficient to show that $\cA(\Sigma_{g, n})_{\ZZ_{2}}$ is not finitely generated, as there is a surjective morphism $\cA(\Sigma_{g, n}) \to \cA(\Sigma_{g, n})_{\ZZ_{2}}$. 

\textsf{Step 2.}
Let $\cC(\Sigma_{g, n})^{+} \subset \cC(\Sigma_{g, n})$ be a subalgebra generated by arcs, loops, and vertices, but not the inverses of vertices. By Definition \ref{def:rholocal}, we know that the homomorphism $\rho : \cA(\Sigma_{g, n}) \to \cC(\Sigma_{g, n})$ indeed factors through $\cC(\Sigma_{g, n})^{+}$. By taking the tensor product with $\ZZ_{2}$, we obtain a homomorphism 
\[
	\rho : \cA(\Sigma_{g, n})_{\ZZ_{2}} \to \cC(\Sigma_{g, n})_{\ZZ_{2}}^{+}.
\]
We have a similar variation for the map $\psi : \cC(\Sigma_{g, n+1})_{\ZZ_{2}}^{+} \to \cC(\Sigma_{g, n})_{\ZZ_{2}}^{+}$. 

\textsf{Step 3.} We specialize to $(g, n) = (1, 1)$. For the vertex $v$ that is forgotten by $\iota : \Sigma_{1, 2} \to \Sigma_{1, 1}$, we impose the associated $\ZZ$-grading structure on $\cC(\Sigma_{1, 2})_{\ZZ_{2}}$, $\cC(\Sigma_{1, 2})_{\ZZ_{2}}^{+}$, and on $\cA(\Sigma_{1, 2})_{\ZZ_{2}}$ (Remark \ref{rem:Zgrading}). Consider the composition 
\[
	\cA(\Sigma_{1, 2})_{\ZZ_{2}, \mathbf{0}} \to \cC(\Sigma_{1, 2})_{\ZZ_{2}, \mathbf{0}}^{+} \to \cC(\Sigma_{1, 1})_{\ZZ_{2}}^{+}
\]
and denote it by $\psi$. 

We claim that the image of $\psi : \cA(\Sigma_{1, 2})_{\ZZ_{2}} \to \cC(\Sigma_{1, 1})_{\ZZ_{2}}^{+}$ is $\cA(\Sigma_{1, 1})_{\ZZ_{2}}$. Indeed, if we denote the unique vertex by $w$, then the image of $\psi$ is generated by $\underline{\beta}$ and $w^{2}\underline{\beta}$ for an ordinary arc $\beta$. Applying the puncture-skein relation, we have $w\underline{\beta} = \underline{\gamma_{1}} + \underline{\gamma_{2}}$ for two loops. But any loop in $\Sigma_{1, 1}$ is a $(p, q)$-torus knot for two relatively prime integers $p$ and $q$, and $\underline{\gamma_{1}} = \underline{\gamma_{2}}$ because they are realized by the same $(p, q)$. Thus, $w\underline{\beta} = 2\underline{\gamma_{1}} = 0 \in \cC(\Sigma_{1, 1})_{\ZZ_{2}}^{+}$ and so is $w^{2}\underline{\beta}$. Therefore, the image of $\psi$ is generated by ordinary arcs only, so $\im\;\psi = \cA(\Sigma_{1, 1})_{\ZZ_{2}}$ is not finitely generated by Proposition \ref{prop:onepuncturenonfinitegeneration}. Therefore, $\cA(\Sigma_{1, 2})_{\ZZ_{2}, \mathbf{0}}$ and $\cA(\Sigma_{1, 2})_{\ZZ_{2}}$ are not finitely generated by Lemma \ref{lem:Nagata}. By Proposition \ref{prop:reduction}, $\cA(\Sigma_{1, n})$ for all $n \ge 1$ are not finitely generated. 

\textsf{Step 4.} 
Now we consider $g \ge 2$. Recall that there is a $g$ to one branched covering $\tau : \Sigma_{g} \to \Sigma_{1}$, branched at two points. This induces a covering map $\tau : \Sigma_{g, 2} \to \Sigma_{1, 2}$ that sends two puctures to the corresponding two punctures. by taking the image of every curve class, we obtain a map $\tau_{*} : \cC(\Sigma_{g, 2})_{R} \to \cC(\Sigma_{1, 2})_{R}$, which is clearly surjective. This map induces a surjective map $\tau_{*} : \cA(\Sigma_{g, 2})_{R} \to \cA(\Sigma_{1, 2})_{R}$. Since $\cA(\Sigma_{1, 2})_{\ZZ_{2}}$ is not finitely generated and $\tau_{*}$ is surjective, $\cA(\Sigma_{g, 2})_{\ZZ_{2}}$ is also not finitely generated. Applying Proposition \ref{prop:reduction}, we get the desired result. 
\end{proof}

\begin{remark}
Another way to think about the special property of $\Sigma_{1, 1}$ is the following. For a fixed triangulation $\cT$, one may write the vertex class $w$ as a Laurent polynomial with respect to the edges in $\cT$. An explicit formula can be found, for example, in \cite[Definition 5.2]{MusikerSchifflerWilliams11}. $\Sigma_{1, 1}$ is the only case that $v$ is a multiple of two. 
\end{remark}

\subsection{Finite generation for $g = 0$}

The situation is entirely different when $g = 0$. The finite generation of $\cA(\Sigma_{0, n})$ follows immediately from the presentation of $\cC(\Sigma_{0, n})$ in Theorem \ref{thm:presentationC0n}. 

\begin{proof}[Proof of Theorem \ref{thm:infinitegenerationA} for $g = 0$]
The proof is essentially identical to that of \cite[Prop 3.2]{ACDHM21}, but for the reader's convenience, we sketch the proof here. 

Recall that, without loss of generality, we assume that the $n$ punctures lie on a small circle $C \subset S^{2}$, and $\underline{\beta_{i,j}}$ is the simple arc connecting $v_{i}$ and $v_{j}$  in the disk bounded by $C$. 

Let $\alpha \in \cA(\Sigma_{0, n})$ be a tagged arc.  So $\underline{\alpha}$ connects two (not necessarily different) punctures. If $\underline{\alpha}$ is inside of $C$, then $\underline{\alpha}$ is isotopic to one of $\underline{\beta_{ij}}$ (if $\underline{\alpha}$ connects two distinct vertices) or $0$ (if two ends of $\underline{\alpha}$ are the same). So $\alpha$ is either zero or one of $\underline{\beta_{ij}}$, $v_{i}\underline{\beta_{ij}}$, $v_{j}\underline{\beta_{ij}}$, or $v_{i}v_{j}\underline{\beta_{ij}}$, depending on the tagging. 

If $\underline{\alpha}$ is outside of $C$, then we can `drag into' $\underline{\alpha}$ and use the puncture-skein relation to break the curve at the vertices.  Then we can describe $\underline{\alpha}$ as a combination of tagged arcs which meet the outside smaller number of times.  Now we may apply induction and get the desired result. 
\end{proof}

We believe that by the virtue of Theorem \ref{thm:presentationC0n}, the following is an interesting and approachable problem. 

\begin{question}
Find a presentation of $\cA(\Sigma_{0, n})$. 
\end{question}

\subsection{Comparison with the upper cluster algebra}

We finish this paper with some remarks on the upper cluster algebra  $\cU(\Sigma_{g, n})$. Recall that $\cC(\Sigma_{g, n})'$ is the subalgebra of $\cC(\Sigma_{g, n})$ generated by isotopy classes of loops, arcs and decorated arcs (Definition \ref{def:Cprime}). 

\begin{lemma}\label{lem:inclusion}
There are inclusions of algebras
\[
	\cA(\Sigma_{g,n}) \subset \cC(\Sigma_{g, n})' \subset \cU(\Sigma_{g, n}).
\]
\end{lemma}

\begin{proof}
The Compatibility Lemma and the fact that the image of $\rho$ factor through $\cC(\Sigma_{g, n})'$ imply the first inclusion. There are two extra classes of generators of $\cC(\Sigma_{g, n})'$ in $\cC(\Sigma_{g, n})' \setminus \cA(\Sigma_{g, n})$: loop classes and $v\underline{\beta}$, where $\underline{\beta}$ is an arc class with two ends both at $v$. (Note that $\underline{\beta}$ and $v^{2}\underline{\beta}$ are in $\cA(\Sigma_{g, n})$, if $n \ge 2$.) We obtain  that  $v\underline{\beta}$ is a sum of two loop classes by applying the puncture-skein relation. For an ordinary triangulation $\cT$ with edge set $E$, it has been proven several times (\cite[Section 12]{FockGoncharov06}, \cite[Theorem 4.2]{MusikerWilliams13}, and \cite[Theorem 3.22]{RogerYang14}) that a loop class is a Laurent polynomial with respect to the edges in a triangulation. The case of a tagged triangulation $\cT^{\bowtie}$ is reduced to the case of an ordinary triangulation, by \cite[Proposition 3.15]{MusikerSchifflerWilliams11}. Thus, we conclude that any element in $\cC(\Sigma_{g, n})'$ can be written as a Laurent polynomial with respect to the edges in a fixed tagged ideal triangulation. 

If we show that this expression is unique,  then set theoretically, $\cC(\Sigma_{g,n})' \subset \cU(\Sigma_{g, n})$ and we are done. This is because the three rings in the statement share isomorphic field of fractions. For a nonzero element $\alpha \in  \cC(\Sigma_{g, n})'$, if there are two Laurent polynomial expressions $f$ and $g$ for $\alpha$, then $f - g$ provides an algebraic relation in their field of fractions generated by edge classes in a fixed triangulation. Since their field of fractions are purely transcendentally generated by edge classes, this is impossible.
\end{proof}

\begin{proof}[Proof of Theorem \ref{thm:torusupper}]
Suppose that $\cA(\Sigma_{g, n}) =\cU(\Sigma_{g, n})$. Lemma \ref{lem:inclusion} implies that $\cA(\Sigma_{g, n}) = \cC(\Sigma_{g, n})'$. By Theorem \ref{thm:Cprimefinitegeneration}, $\cC(\Sigma_{g, n})'$ is finitely generated, but $\cA(\Sigma_{g, n})$ is not finitely generated.
\end{proof}

\begin{remark}\label{rem:g=0AC}
When $g = 0$, all loop classes are generated by tagged arc classes, as proved in \cite[Proposition 2.2]{BKPW16Involve} and  as evidenced by  Theorem~\ref{thm:presentationC0n}. Thus, $\cA(\Sigma_{0, n}) = \cC(\Sigma_{0, n})'$. 
\end{remark}

\begin{remark}\label{rem:CandU}
If $n \ge 2$, $\cC(\Sigma_{g, n})$ is not a subalgebra of $\cU(\Sigma_{g, n})$, because of the vertex classes. For a fixed ordinary triangulation $\cT$ and its edge set $E = \{e_{i}\}$, a vertex class $v$ can be written as a Laurent polynomial with respect to $E$ (see the proof of \cite[Lemma 3.2]{MoonWong21}). However, this is no longer true for a tagged triangulation $\cT^{\bowtie}$. On the other hand, when $n = 1$, we do not consider a tagged triangulation, so $v \in \cU(\Sigma_{g, n})$ and hence $\cC(\Sigma_{g, n}) \subset \cU(\Sigma_{g, n})$. 
\end{remark}

\begin{remark}\label{rem:midalgebra}
In a recent breakthrough in \cite{GrossHackingKeelKontsevich18}, for each combinatorial data defining a cluster algebra, Gross, Hacking, Keel, and Kontsevich defined yet another algebra motivated from mirror symmetry, the so-called \emph{mid-cluster algebra} ($\mathrm{mid}(V)$ in their terminology). For $\cA(\Sigma_{g, n})$, the mid-algebra is indeed equal to $\cC(\Sigma_{g, n})'$ and it admits a canonical basis parametrized by the tropical points of the dual cluster variety (\cite[Theorem 1.3]{MandelQin23}, \cite[Section 12]{FockGoncharov06}). 
\end{remark}

To the authors' knowledge, it has not been rigorously proved whether $\cC(\Sigma_{g, n})' = \cU(\Sigma_{g, n})$ or not. 

\begin{conjecture}\label{conj:CandU}
\begin{enumerate}
\item $\cC(\Sigma_{g, 1}) = \cU(\Sigma_{g, 1})$. 
\item If $n \ge 2$, $\cC(\Sigma_{g, n})' = \cU(\Sigma_{g, n})$. In particular, if $n \ge 4$, $\cA(\Sigma_{0, n}) = \cU(\Sigma_{0, n})$. 
\end{enumerate}
\end{conjecture}

\bibliographystyle{alpha}

\end{document}